\documentclass[article,12pt]{amsart}
\usepackage{amsfonts}
\usepackage{amsmath}
\usepackage{graphicx}
\usepackage{color}
\usepackage{epsfig}

\numberwithin{equation}{section}
\theoremstyle{plain}
\newtheorem{thm}{Theorem}[section]
\newtheorem{rem}{Remark}[section]

\newtheorem{lem}{Lemma}[section]

\newcommand{\dE}{\mathbb{E}}
\newcommand{\dR}{\mathbb{R}}

\newcommand{\cA}{\mathcal{A}}

\newcommand{\cN}{\mathcal{N}}
\newcommand{\cF}{\mathcal{F}}
\newcommand{\cH}{\mathcal{H}}
\newcommand{\cR}{\mathcal{R}}

\newcommand{\veps}{\varepsilon}
\newcommand{\wh}{\widehat}
\newcommand{\wt}{\widetilde}
\newcommand{\rI}{\mathrm{I}}
\newcommand{\ind}{\mbox{1}\kern-.25em \mbox{I}}
\font\calcal=cmsy10 scaled\magstep1
\def\build#1_#2^#3{\mathrel{\mathop{\kern 0pt#1}\limits_{#2}^{#3}}}
\def\liml{\build{\longrightarrow}_{}^{{\mbox{\calcal L}}}}
\def\limp{\build{\longrightarrow}_{}^{{\mbox{\calcal P}}}}
\def\videbox{\mathbin{\vbox{\hrule\hbox{\vrule height1ex \kern.5em
\vrule height1ex}\hrule}}}

\email{Bernard.Bercu@math.u-bordeaux1.fr}
\email{Frederic.Proia@inria.fr}

\begin{document}

\title[On the asymptotic behavior of the Durbin-Watson statistic]
{A sharp analysis on the asymptotic behavior of the Durbin-Watson statistic
for the first-order autoregressive process
\vspace{2ex}}
\author{Bernard Bercu}
\address{Universit\'e Bordeaux 1, Institut de Math\'ematiques de Bordeaux,
UMR 5251, and INRIA Bordeaux, team ALEA, 351 Cours de la Lib\'eration, 33405 Talence cedex, France.}
\author{Fr\'ed\'eric Proia}
\address{
}
\thanks{}

\begin{abstract}
The purpose of this paper is to provide a sharp analysis on the asymptotic behavior of the Durbin-Watson statistic. 
We focus our attention on the first-order autoregressive process where the driven noise is also given by
a first-order autoregressive process. We establish the almost sure convergence and the asymptotic normality 
for both the least squares estimator of the unknown parameter of the autoregressive process as well as 
for the serial correlation estimator associated to the driven noise. In addition, the almost sure rates 
of convergence of our estimates are also provided. It allows us to establish the almost sure convergence 
and the asymptotic normality for the Durbin-Watson statistic. 
Finally, we propose a new bilateral statistical test for residual autocorrelation.
\end{abstract}

\maketitle


\section{INTRODUCTION}


The Durbin-Watson statistic is very well-known in Econometry and Statistics. 
It was introduced by the pioneer works of Durbin and Watson \cite{DurbinWatson50}, \cite{DurbinWatson51}, \cite{DurbinWatson71}, in order to test the serial independence of the driven noise of a linear regression model. The statistical test based on the Durbin-Watson statistic works pretty well for linear regression models, and its power was investigated by Tillman \cite{Tillman75}. However, as it was observed by Malinvaud \cite{Malinvaud61} and Nerlove and Wallis \cite{NerloveWallis66}, its widespread use in inappropriate situations may lead to inadequate conclusions. More precisely, for linear regression models containing lagged dependent random variables, the Durbin-Watson statistic may be asymptotically biased. In order to prevent this misuse, Durbin \cite{Durbin70} proposed alternative tests based on the redesign of the original one. Then, he explained how to use them in the particular case of the first-order autoregressive process previously investigated in \cite{Malinvaud61} and \cite{NerloveWallis66}. Maddala and Rao \cite{MaddalaRao73} and Park \cite{Park75} showed by simulations that alternative tests significantly outperform the inappropriate one even on small-sized samples. Inder \cite{Inder84}, \cite{Inder86} and Durbin \cite{Durbin86} went even deeper in the approximation of the critical values and distributions of the alternative tests under the null hypothesis. Afterwards, additional improvements were brought by King and Wu \cite{KingWu91} and more recently, Stocker \cite{Stocker06} gave substantial contributions to the study of the asymptotic bias in the Durbin-Watson statistic resulting from the presence of lagged dependent random variables.
\\ \vspace{-1ex} \par

Our purpose is to investigate several open questions left unanswered during four decades on the Durbin-Watson statistic \cite{Durbin70}, \cite{Durbin86}, \cite{NerloveWallis66}. We shall focus our attention on the first-order autoregressive process given, for all $n\geq 1$, by 
\begin{equation}  
\label{AR}
\vspace{1ex}
\left\{
\begin{array}[c]{ccccc}
X_{n} & = & \theta X_{n-1} & + & \veps_n \vspace{1ex}\\
\veps_n & = & \rho \veps_{n-1} & + & V_n 
\end{array}
\right.
\end{equation}
where the unknown parameters $\vert \theta \vert < 1$, $\vert \rho \vert < 1$.
Via an extensive use of the theory of martingales \cite{Duflo97}, \cite{HallHeyde80}, we shall provide a sharp and rigorous analysis on the asymptotic behavior of the least squares estimators of $\theta$ and $\rho$. The previous results of convergence were only established in probability \cite{Malinvaud61}, \cite{NerloveWallis66}. We shall prove the almost sure convergence as well as the asymptotic normality of the least squares estimators of $\theta$ and $\rho$. We will deduce the almost sure convergence and the asymptotic normality for the Durbin-Watson statistic. Therefore, we shall be in the position to propose a new bilateral test for residual autocorrelation under the null hypothesis as well as under the alternative hypothesis.
\\ \vspace{1ex} \par
The paper is organized as follows. Section 2 is devoted to the estimation of the autoregressive parameter. We establish the almost sure convergence of the least squares estimator $\widehat{\theta}_{n}$ to the limiting value 
\begin{equation}
\label{thetastar}
\theta^{*}= \frac{\theta + \rho}{1 + \theta\rho}.
\end{equation}
One can observe that $\theta^{*}=\theta$ if and only if $\rho=0$. The asymptotic normality of 
$\widehat{\theta}_{n}$ as well as the quadratic strong law and the law of iterated logarithm are also provided. 
Section 3 deals with the estimation of the serial correlation parameter. We prove the almost sure convergence of the least squares estimator $\widehat{\rho}_{n}$ to
\begin{equation}
\label{rhostar}
\vspace{1ex}
\rho^{*}=\theta \rho \theta^{*}= \frac{\theta \rho(\theta + \rho)}{1 + \theta\rho}.
\end{equation}
As before, the asymptotic normality of $\widehat{\rho}_{n}$, the quadratic strong law and the law of iterated logarithm are also provided. It allows us to establish in Section 4 the almost sure convergence of the Durbin-Watson statistic $\widehat{D}_{n}$ to 
\begin{equation}
\label{Dstar}
\vspace{1ex}
D^{*}= 2(1-\rho^{*})
\end{equation}
together with its asymptotic normality. Our sharp analysis on the asymptotic behavior of $\widehat{D}_{n}$ 
is true whatever the values of the parameters $\theta$ and $\rho$ inside the interval $] \! -1,1[$. 
Consequently, we are able in Section 4 to propose a new bilateral statistical test for residual autocorrelation. 
A short conclusion is given in Section 5. 
All the technical proofs of Sections 2, 3, and 4 are postponed in Appendices A, B, and C, respectively.


\section{ON THE AUTOREGRESSIVE PARAMETER}


Consider the first-order autoregressive process given by \eqref{AR} where the initial values
 $X_0$ and $\veps_0$ may be arbitrarily chosen. In all the sequel,  we assume that $(V_{n})$ is a sequence of square-integrable, independent 
 and identically distributed random variables with zero mean and variance $\sigma^2 > 0$. In order to estimate the unknown
 parameter $\theta$, it is natural to make use of the least squares estimator $\wh{\theta}_n$ which minimizes
 $$
 \Delta_n(\theta)=\sum_{k=1}^n (X_k - \theta X_{k-1})^2.
 $$
Consequently, we obviously have for all $n \geq 1$,
\begin{equation}  
\label{THETA_EST}
\wh{\theta}_{n} = \frac{\sum_{k=1}^{n} X_{k} X_{k-1}}{\sum_{k=1}^{n} X_{k-1}^2}.
\end{equation}

Our first result concerns the almost sure convergence of $\wh{\theta}_{n}$ to the limiting value $\theta^{*}$
given by \eqref{thetastar}. One can observe that the convergence in probability of $\wh{\theta}_{n}$ to $\theta^{*}$
was already proven in \cite{Malinvaud61}, \cite{NerloveWallis66}. We improve this previous result by establishing
the almost sure convergence of $\wh{\theta}_{n}$ to $\theta^{*}$.

\begin{thm}
\label{THM_ASCVGTHETA}
We have the almost sure convergence
\begin{equation}  
\label{ASCVGTHETA}
\lim_{n\rightarrow \infty} \wh{\theta}_{n} = \theta^{*} \hspace{1cm} \textnormal{a.s.}
\end{equation}
\end{thm}

Our second result deals with the asymptotic normality of $\wh{\theta}_{n}$ where we denote
\begin{equation}
\label{VARTHETA}
\sigma^2_\theta=\frac{(1-\theta^2)(1-\theta\rho)(1-\rho^2)}{(1+\theta\rho)^3}.
\end{equation}

\begin{thm}
\label{THM_CLTTHETA}
Assume that $(V_{n})$ has a finite moment of order  $4$. Then, we have the asymptotic normality
\begin{equation}  
\label{CLTTHETA}
\sqrt{n} \left( \wh{\theta}_{n} - \theta^{*} \right) \liml \cN ( 0, \sigma^2_\theta).
\end{equation}
\end{thm}

\begin{rem}
In the  well-known case where the residuals are not correlated, which means that $\rho = 0$, we clearly have $\theta^{*} = \theta$,
$\sigma^2_\theta= 1- \theta^2$ and we find again the asymptotic normality
$$ \sqrt{n}\left( \wh{\theta}_{n} - \theta \right) \liml \cN ( 0, 1-\theta^2). $$
\end{rem}

After establishing the almost sure convergence of the estimator $\wh{\theta}_{n}$ and its asymptotic normality, we focus our attention
on the almost sure rates of convergence. 

\begin{thm}
\label{THM_LILTHETA}
Assume that $(V_{n})$ has a finite moment of order  $4$. Then, we have the quadratic strong law
\begin{equation}  
\label{QSLTHETA}
\lim_{n\rightarrow \infty} \frac{1}{\log n} \sum_{k=1}^{n} \left( \wh{\theta}_{k} - \theta^{*} \right)^2=\sigma^2_\theta
\hspace{1cm} \textnormal{a.s.}
\end{equation}
 where $\sigma^2_\theta$ is given by \eqref{VARTHETA}. In addition, we also have the law of iterated logarithm
\begin{eqnarray}  
\limsup_{n \rightarrow \infty} \left(\frac{n}{2 \log \log n} \right)^{1/2}
\left(\wh{\theta}_{n}-\theta^{*} \right)
&=& - \liminf_{n \rightarrow \infty}
\left(\frac{n}{2 \log\log n}\right)^{1/2} 
\left(\wh{\theta}_{n}-\theta^{*}\right)  \notag \\
&=& \sigma_\theta
\hspace{1cm}\textnormal{a.s.}
\label{LILTHETA}
\end{eqnarray}
Consequently, 
\begin{equation}
\label{LILSUPTHETA}  
\limsup_{n \rightarrow \infty} \left(\frac{n}{2 \log \log n} \right)
\left(\wh{\theta}_{n}-\theta^{*} \right)^2=\sigma^2_\theta
\hspace{1cm}\textnormal{a.s.}
\end{equation} 
\end{thm}

\begin{proof} The proofs are given in Appendix A. 
\end{proof}

\begin{rem}
It clearly follows from \eqref{LILSUPTHETA} that
\begin{equation}
\label{ASRATETHETA}  
\left( \wh{\theta}_{n}-\theta^{*} \right)^2= O\left( \frac{\log \log n}{n}\right)
\hspace{1cm}\textnormal{a.s.}
\end{equation} 
This almost sure rate of convergence will be useful in all the sequel.
\end{rem}


\section{ON THE SERIAL CORRELATION PARAMETER}


This section is devoted to the estimation of the serial correlation parameter $\rho$.  First of all, it is necessary
to evaluate, at step $n$, the least squares residuals given, for all $1\leq k \leq n$, by
\begin{equation}  
\label{RHO_EST_RES}
\wh{\veps}_{k} = X_{k} - \wh{\theta}_{n} X_{k-1}.
\end{equation}
The initial value $\wh{\veps}_0$ may be arbitrarily chosen
and we take $\wh{\veps}_0 = X_0$.
Then, a natural way to estimate $\rho$ is to make use of the least squares estimator 
\begin{equation}  
\label{RHO_EST}
\wh{\rho}_{n} = \frac{\sum_{k=1}^{n} \wh{\veps_{k}} \wh{\veps}_{k-1}}{\sum_{k=1}^{n} \wh{\veps}_{k-1}^{\, \, 2}}.
\end{equation}
The asymptotic behavior of $\wh{\theta}_{n}$ and $\wh{\rho}_{n}$ are quite similar.
However, one can realize that the results of this section are much more tricky to establish than those of the previous one.
We first state the almost sure convergence of $\wh{\rho}_{n}$ to $\rho^{*}$.

\begin{thm}
\label{THM_ASCVGRHO}
We have the almost sure convergence
\begin{equation}  
\label{ASCVGRHO}
\lim_{n\rightarrow \infty} \wh{\rho}_{n} = \rho^{*} \hspace{1cm} \textnormal{a.s.}
\end{equation}
\end{thm}

Our next result deals with the joint asymptotic normality of $\wh{\theta}_{n}$ and $\wh{\rho}_{n}$. Denote
\begin{equation}
\label{VARRHO}
\sigma^2_\rho=\frac{(1-\theta \rho)}{(1+\theta\rho)^3}\big((\theta+\rho)^2(1+\theta \rho)^2+(\theta \rho)^2 (1-\theta^2)(1-\rho^2)\big).
\end{equation}
In addition, let $\Gamma$ be the semi-definite positive covariance matrix given by
\begin{equation}
\label{DEFGAMMA}
\Gamma= \begin{pmatrix}
\sigma^2_\theta & \ \theta \rho \sigma^2_\theta \\
\theta \rho \sigma^2_\theta & \sigma^2_\rho
\end{pmatrix}.
\end{equation}

\begin{thm}
\label{THM_CLTRHO}
Assume that $(V_{n})$ has a finite moment of order  $4$. Then, we have the joint asymptotic normality
\begin{equation}  
\label{CLTTHETARHO}
\sqrt{n} \begin{pmatrix}
\ \wh{\theta}_{n} - \theta^{*}  \\
\ \wh{\rho}_{n} - \rho^{*}
\end{pmatrix}
 \liml \cN \big( 0, \Gamma \big).
\end{equation}
In particular,
\begin{equation}  
\label{CLTRHO}
\sqrt{n} \Big( \wh{\rho}_{n} - \rho^{*}\Big)
 \liml \cN ( 0, \sigma^2_\rho).
\end{equation}
\end{thm}

\begin{rem}
\label{REM_GAMMA}
The covariance matrix $\Gamma$ is invertible if and only if $\theta\neq - \rho$ since one can see by a straightforward calculation that
$$
\det(\Gamma)=  \frac{\sigma^2_\theta (\theta+\rho)^2(1-\theta \rho)}{(1+\rho^2)}.
$$
Moreover, in the particular case where $\theta=- \rho$, 
$$
\sqrt{n}\ \wh{\theta}_{n}  \liml \cN \left( 0, \frac{1+\theta^2}{1-\theta^2}\right)
\hspace{1cm} \text{and} \hspace{1cm}
\sqrt{n}\ \wh{\rho}_{n}  \liml \cN \left( 0, \frac{\theta^4(1+\theta^2)}{1-\theta^2}\right).
$$
Finally, if the residuals are not correlated which means that $\rho = 0$, 
$$ \sqrt{n}\ \wh{\rho}_{n}  \liml \cN \left( 0, \theta^2 \right). $$
\end{rem}

The almost sure rates of convergence for $\wh{\rho}_{n}$ are as follows.

\begin{thm}
\label{THM_LILRHO}
Assume that $(V_{n})$ has a finite moment of order  $4$. Then, we have the quadratic strong law
\begin{equation}  
\label{QSLRHO}
\lim_{n\rightarrow \infty} \frac{1}{\log n} \sum_{k=1}^{n} \Big( \wh{\rho}_{k} - \rho^{*} \Big)^2=\sigma^2_\rho
\hspace{1cm} \textnormal{a.s.}
\end{equation}
 where $\sigma^2_\rho$ is given by \eqref{VARRHO}. In addition, we also have the law of iterated logarithm
\begin{eqnarray}  
\limsup_{n \rightarrow \infty} \left(\frac{n}{2 \log \log n} \right)^{1/2}
\Big(\wh{\rho}_{n}-\rho^{*} \Big)
&=& - \liminf_{n \rightarrow \infty}
\left(\frac{n}{2 \log\log n}\right)^{1/2} 
\Big(\wh{\rho}_{n}-\rho^{*}\Big)  \notag \\
&=& \sigma_\rho
\hspace{1cm}\textnormal{a.s.}
\label{LILRHO}
\end{eqnarray}
Consequently, 
\begin{equation}
\label{LILSUPRHO}  
\limsup_{n \rightarrow \infty} \left(\frac{n}{2 \log \log n} \right)
\Big(\wh{\rho}_{n}-\rho^{*} \Big)^2=\sigma^2_\rho
\hspace{1cm}\textnormal{a.s.}
\end{equation} 
\end{thm}

\begin{proof} The proofs are given in Appendix B. 
\end{proof}

\begin{rem}
We obviously deduce from \eqref{LILSUPRHO} that
\begin{equation}
\label{ASRATERHO}  
\Big( \wh{\rho}_{n}-\rho^{*} \Big)^2= O\left( \frac{\log \log n}{n}\right)
\hspace{1cm}\textnormal{a.s.}
\end{equation} 
\end{rem}

The estimators $\wh{\theta}_{n}$ and $\wh{\rho}_{n}$ are self-normalized. Consequently, the asymptotic variances
$\sigma^2_\theta$ and $\sigma^2_\rho$ do not depend on the variance $\sigma^2$ associated to the driven noise $(V_n)$.
We now focus our attention on the estimation of $\sigma^2$. The estimator $\wh{\rho}_{n}$ allows us to evaluate, at step $n$, the 
least squares residuals given, for all $1 \leq k \leq n$, by
$$
\wh{V}_{k} = \wh{\veps}_{k} - \wh{\rho}_{n} \wh{\veps}_{k-1}.
$$
Then, we propose to make use of 
\begin{equation*}  
\label{SIG_EST}
\wh{\sigma}_{n}^2 = \frac{1}{n} \sum_{k=1}^{n} \wh{V}_{k}^2 .
\end{equation*}
We have the almost sure convergence
\begin{equation}  
\label{ASCVGSIG}
\lim_{n\rightarrow \infty} \wh{\sigma}_{n} ^2= 
\frac{\sigma^2 \big((1 + \theta\rho)^2 - (\theta \rho)^2 (\theta + \rho)^2 \big)}{(1 - \theta \rho)(1 + \theta\rho)^3}
\hspace{1cm} \textnormal{a.s.}
\end{equation}
The proof is left to the reader as it follows essentially the same lines as that of \eqref{ASCVGRHO}.


\section{ON THE DURBIN-WATSON STATISTIC}


We shall now investigate the asymptotic behavior of the Durbin-Watson statistic 
\cite{DurbinWatson50}, \cite{DurbinWatson51}, \cite{DurbinWatson71} 
given, for all $n\geq 1$, by
\begin{equation}  
\label{DW_STAT}
\wh{D}_{n} = \frac{\sum_{k=1}^{n} (\wh{\veps}_{k} - \wh{\veps}_{k-1})^2}{\sum_{k=0}^{n} \wh{\veps}_{k}^{\,2}}.
\end{equation}
One can observe that $\wh{D}_{n}$ and $\wh{\rho}_{n}$ are asymptotically linked together by an affine transformation.
Consequently, the results of the previous section allow us to establish the asymptotic behavior of $\wh{D}_{n}$.
We start with the almost sure convergence of $\wh{D}_{n}$ to $D^{*}$.

\begin{thm}
\label{THM_ASCVGDW}
We have the almost sure convergence
\begin{equation}  
\label{ASCVGDW}
\lim_{n\rightarrow \infty} \wh{D}_{n} = D^{*} \hspace{1cm} \textnormal{a.s.}
\end{equation}
\end{thm}

Our next result deals with the asymptotic normality of $\wh{D}_{n}$. It will be the keystone of a new bilateral statistical test deciding in particular, 
for a given significance level, whether residuals are autocorrelated or not. Denote $\sigma^2_D=4 \sigma^2_\rho$ where the
variance $\sigma^2_\rho$ is given by \eqref{VARRHO}.

\begin{thm}
\label{THM_CLTDW}
Assume that $(V_{n})$ has a finite moment of order  $4$. Then, we have the asymptotic normality
\begin{equation}  
\label{CLTDW}
\sqrt{n} \left( \wh{D}_{n} - D^{*} \right) \liml \cN ( 0, \sigma^2_D).
\end{equation}
\end{thm}

\begin{rem}
We immediately deduce from \eqref{CLTDW} that
\begin{equation}
\label{CVGXI2}
\frac{n}{\sigma_{D}^2} \left( \wh{D}_{n} - D^{*} \right)^2 \liml \chi^2
\end{equation}
where $\chi^2$ has a Chi-square distribution with one degree of freedom.
\end{rem}

Before providing our statistical test, we focus our attention on the almost sure rates of convergence for $\wh{D}_{n}$ 
which are based on the asymptotic linear relation between $\wh{D}_{n}$ and $\wh{\rho}_{n}$.

\begin{thm}
\label{THM_LILDW}
Assume that $(V_{n})$ has a finite moment of order  $4$. Then, we have the quadratic strong law
\begin{equation}  
\label{QSLDW}
\lim_{n\rightarrow \infty} \frac{1}{\log n} \sum_{k=1}^{n} \left( \wh{D}_{k} - D^{*} \right)^2=\sigma^2_D
\hspace{1cm} \textnormal{a.s.}
\end{equation}
In addition, we also have the law of iterated logarithm
\begin{eqnarray}  
\limsup_{n \rightarrow \infty} \left(\frac{n}{2 \log \log n} \right)^{1/2}
\left(\wh{D}_{n}-D^{*} \right)
&=& - \liminf_{n \rightarrow \infty}
\left(\frac{n}{2 \log\log n}\right)^{1/2} 
\left(\wh{D}_{n}-D^{*}\right)  \notag \\
&=& \sigma_D
\hspace{1cm}\textnormal{a.s.}
\label{LILDW}
\end{eqnarray}
Consequently, 
\begin{equation}
\label{LILSUPDW}  
\limsup_{n \rightarrow \infty} \left(\frac{n}{2 \log \log n} \right)
\left(\wh{D}_{n}-D^{*} \right)^2=\sigma^2_D
\hspace{1cm}\textnormal{a.s.}
\end{equation} 
\end{thm}

We are now in the position to propose our new  bilateral statistical test built on the Durbin-Watson statistic $\wh{D}_n$. First of all, 
we shall not investigate the trivial case $\theta=0$ since our statistical test procedure is of interest only for autoregressive processes.
In addition, we shall note the existence of a critical case as introduced in Remark \ref{REM_GAMMA}. 
Indeed, if $\theta = -\rho$, the covariance matrix $\Gamma$ given by
\eqref{DEFGAMMA} is not invertible and the distribution of the statistic associated to the test we plan to establish will be degenerate. For this reason, we suggest a preliminary test for the hypothesis $``\theta = -\rho"$, allowing us to switch from one test to another if necessary. More precisely, we first wish to test

$$\cH_0\,:\,`` \theta = -\rho" \hspace{1cm} \text{against}\hspace{1cm}\cH_1\,:\,`` \theta \neq -\rho".
\vspace{2ex}
$$

Under the null hypothesis $\cH_0$, it is easy to see that $D^{*}=2$. According to Remark \ref{REM_GAMMA}, we have
\begin{equation}
\label{CRITCASE_CLT}
\frac{n(1-\theta^2)}{4\theta^4 (1+\theta^2)} \left( \wh{D}_{n} - 2 \right)^2 \liml \chi^2
\end{equation}
where $\chi^2$ has a Chi-square distribution with one degree of freedom. Moreover, the model can be rewritten under $\cH_0$, for all $n \geq 2$, as
\begin{equation}
\label{CRITCASE_MOD}
X_{n} = \theta^2 X_{n-2} + V_{n}.
\end{equation}
Then, we propose to make use of the standard least squares estimator $\wh{\vartheta}_n^{\,2}$ of $\theta^2$ 
\begin{equation}
\label{CRITCASE_EST}
\wh{\vartheta}_n^{\,2} = \frac{\sum_{k=2}^{n} X_{k-2}X_{k}}{\sum_{k=2}^{n} X_{k-2}^2}.
\end{equation}
Under $\cH_0$, we have the almost sure convergence of $\wh{\vartheta}_n^{\,2}$ to $\theta^2$. In addition, we obviously have
$D^{*} \neq 2$ under $\cH_1$. These results under the null and the alternative hypothesis lead to Theorem \ref{THM_CRITCASE}, 
whose proof immediately follows from \eqref{CRITCASE_CLT}.
\begin{thm}
\label{THM_CRITCASE}
Assume that $(V_{n})$ has a finite moment of order  $4$, $\theta \neq 0$ and $\rho \neq 0$.
Then, under the null hypothesis $\cH_0\,:\,`` \theta = -\rho"$,
\begin{equation}
\label{DWCRITCASEH0}
\frac{n(1-\wh{\vartheta}_n^{\,2})}{4(\wh{\vartheta}_n^{\,2})^2 (1+\wh{\vartheta}_n^{\,2})}  \left( \wh{D}_{n} - 2 \right)^2 \liml \chi^2
\end{equation}
where $\chi^2$ has a Chi-square distribution with one degree of freedom. In addition,
under the alternative hypothesis $\cH_1\,:\,`` \theta \neq -\rho"$,
\begin{equation}
\label{DWCRITCASEH1}
\lim_{n\rightarrow \infty} \frac{n(1-\wh{\vartheta}_n^{\,2})}{4(\wh{\vartheta}_n^{\,2})^2 (1+\wh{\vartheta}_n^{\,2})}  \left( \wh{D}_{n} - 2 \right)^2 = + \infty \hspace{1cm} \textnormal{a.s.}
\end{equation}
\end{thm}
For a significance level $\alpha$ where $0< \alpha <1$, the acceptance and rejection regions are given by
$\cA= [0, z_{\alpha}]$ and $\cR =] z_{\alpha}, +\infty [$
where $z_{\alpha}$ stands for the $(1-\alpha)$-quantile of the Chi-square distribution with one degree of freedom.
The null hypothesis $\cH_0$ will be accepted if the empirical value
$$
\frac{n(1-\wh{\vartheta}_n^{\,2})}{4(\wh{\vartheta}_n^{\,2})^2 (1+\wh{\vartheta}_n^{\,2})}  \left( \wh{D}_{n} - 2 \right)^2 \leq z_{\alpha},
$$
and rejected otherwise. Assume now that we accept $\cH_0$, 
which means that we admit \textit{de facto} the hypothesis $``\theta = -\rho"$.
For a given value $\rho_0$ such that $|\rho_0|< 1$, we wish to test whether or not the serial correlation parameter is equal to $\rho_0$, setting

$$\cH_0\,:\,`` \rho = \rho_0" \hspace{1cm} \text{against}\hspace{1cm}\cH_1\,:\,`` \rho \neq \rho_0".
\vspace{2ex} $$
One shall proceed once again to the test described by Theorem \ref{THM_CRITCASE}, taking $\rho_0^2$ in lieu of $\wh{\vartheta}_{n}^2$, insofar as one can easily agree that our test statistic satisfies the same properties, under $\cH_0$ as under $\cH_1$, by virtue of Remark \ref{REM_GAMMA}.
This alternative solution is necessary to avoid the degenerate situation implied by the critical case $\theta = -\rho$. Let us now focus on the more widespread case where the preliminary test leads to a rejection of $\cH_0$, admitting $``\theta \neq -\rho"$. 
For that purpose, denote $\widetilde{\theta}_n=\wh{\theta}_n + \wh{\rho}_n - \rho_0$ and 
$\widetilde{D}_n = 2 \big( 1 - \widetilde{\rho}_n \big)$ where
\begin{equation}
\label{DEFRHOT}
\widetilde{\rho}_n =\frac{\rho_0 \widetilde{\theta}_n (\widetilde{\theta}_n + \rho_0)}{1 + \rho_0 \widetilde{\theta}_n}.
\end{equation}

\noindent
One can observe that our statistical test procedure works whatever the value of the autoregressive parameter $\theta$
inside the interval $] \! -1,1[$ with $\theta \neq -\rho$. Moreover, it follows from \eqref{ASCVGTHETA} and \eqref{ASCVGRHO} that under the
null hypothesis $\cH_0$,
$$
\lim_{n\rightarrow \infty} \widetilde{\theta}_{n} = \theta + \rho_0 - \rho_0= \theta \hspace{1cm} \text{a.s.}
$$
To construct our statistical test, we need to introduce more notations. Denote
\begin{eqnarray*}
\wh{a}_n&=&-\rho_0 \big(\wh{\theta}_n + \widetilde{\theta}_n)=-\rho_0 \big(2\wh{\theta}_n + \wh{\rho}_n-\rho_0),\\
\wh{b}_n&=&1-\rho_0 \wh{\theta}_n, 
\end{eqnarray*}
and let $\wh{w}_n$ be the vector of $\dR^2$ given by $\wh{w}_n^{\,\prime}=(\wh{a}_{n}, \wh{b}_{n})$.
In addition, let
\begin{equation}
\label{DEFGAMMAN}
\wh{\Gamma}_n= \begin{pmatrix}
\wh{\alpha}_n & \  \rho_0 \widetilde{\theta}_n \wh{\alpha}_n \\
\rho_0 \widetilde{\theta}_n \wh{\alpha}_n & \wh{\beta}_n
\end{pmatrix}.
\end{equation}
where $\wh{\alpha}_n$ and $\wh{\beta}_n$ are defined as
\begin{eqnarray*}
\wh{\alpha}_n &=& \frac{(1-\widetilde{\theta}_n^{\, 2})(1-\rho_0 \widetilde{\theta}_n)(1-\rho_0^2)}{(1+\rho_0 \widetilde{\theta}_n)^3}, \\
 \wh{\beta}_n &=& \frac{(1- \rho_0\widetilde{\theta}_n)}{(1+ \rho_0\widetilde{\theta}_n)^3}
 \big((\widetilde{\theta}_n+ \rho_0)^2(1+\rho_0 \widetilde{\theta}_n)^2+( \rho_0 \widetilde{\theta}_n)^2 (1-\widetilde{\theta}_n^{\, 2})(1-\rho_0^2)\big).
\end{eqnarray*}
Furthermore, denote
\begin{equation}
\label{DEFTAUN}
\wh{\tau}_n^{\, 2}=\frac{4}{(1+\rho_0 \widetilde{\theta}_n)^2} \wh{w}_n^{\,\prime}\wh{\Gamma}_n  \wh{w}_n.
\end{equation}

\begin{thm}
\label{THM_DWTEST}
Assume that $(V_{n})$ has a finite moment of order  $4$, $\theta \neq \rho_0$ and $\theta\neq - \rho$.
Then, under the null hypothesis $\cH_0\,:\,`` \rho = \rho_0"$,
\begin{equation}
\label{DWTESTH0}
\frac{n}{\wh{\tau}_n^{\, 2}} \left( \wh{D}_{n} - \widetilde{D}_{n} \right)^2 \liml \chi^2
\end{equation}
where $\chi^2$ has a Chi-square distribution with one degree of freedom. In addition, 
under the alternative hypothesis $\cH_1\,:\,`` \rho \neq \rho_0"$,
\begin{equation}
\label{DWTESTH1}
\lim_{n\rightarrow \infty}  \frac{n}{\wh{\tau}_n^{\, 2}} \left( \wh{D}_{n} - \widetilde{D}_{n} \right)^2 = + \infty \hspace{1cm} \textnormal{a.s.}
\end{equation}
\end{thm}

One can observe by a symmetry argument on the role played by $\theta$ and $\rho$, that the assumption $\theta = \rho_0$
is not restrictive since the latter
can be seen as another way of expressing Theorem \ref{THM_DWTEST}.
From a practical point of view, for a significance level $\alpha$ where $0<\alpha <1$, the acceptance and rejection regions are given by 
$\cA= [0, z_{\alpha}]$ and $\cR =] z_{\alpha}, +\infty [$
where $z_{\alpha}$ stands for the $(1-\alpha)$-quantile of the Chi-square distribution with one degree of freedom.
The null hypothesis $\cH_0$ will be accepted if the empirical value
$$
\frac{n}{\wh{\tau}_n^{\, 2}} \left( \wh{D}_{n} - \widetilde{D}_{n} \right)^2 \leq z_{\alpha},
$$
and rejected otherwise. Moreover, if one wishes to test

$$\cH_0\,:\,`` \rho = 0" \hspace{1cm} \text{against}\hspace{1cm}\cH_1\,:\,`` \rho \neq 0",
\vspace{2ex}
$$ 
our statistical test procedure is very simple. As a matter of fact, we are in the particular case $\rho_0=0$ which means that
$\widetilde{D}_n = 2$, $\wh{a}_n=0$ and $\wh{b}_n=1$. We can also replace $\widetilde{\theta}_n$
by $\wh{\theta}_n$ so $\wh{\tau}_n^{\, 2}$ reduces to $\wh{\tau}_n^{\, 2}=4\wh{\theta}_n^{\, 2}$.

\begin{thm}
\label{THM_DWTESTBIS}
Assume that $(V_{n})$ has a finite moment of order  $4$, $\theta \neq 0$ and $\theta\neq - \rho$.
Then, under the null hypothesis $\cH_0\,:\,`` \rho = 0"$,
\begin{equation}
\label{DWTESTH0BIS}
\frac{n}{4\wh{\theta}_n^{\, 2}} \left( \wh{D}_{n} - 2 \right)^2 \liml \chi^2
\end{equation}
where $\chi^2$ has a Chi-square distribution with one degree of freedom. In addition, 
under the alternative hypothesis $\cH_1\,:\,`` \rho \neq 0"$,
\begin{equation}
\label{DWTESTH1BIS}
\lim_{n\rightarrow \infty}  \frac{n}{4\wh{\theta}_n^{\, 2}} \left( \wh{D}_{n} - 2 \right)^2 = + \infty \hspace{1cm} \textnormal{a.s.}
\end{equation}
\end{thm}

\begin{proof} The proofs are given in Appendix C. 
\end{proof}


\section{CONCLUSION}


Via an extensive use of the theory of martingales, we have provided a sharp analysis on the asymptotic behavior of the least squares estimators $\wh{\theta}_n$ and $\wh{\rho}_n$ which has allowed us to deduce the asymptotic behavior of the Durbin-Watson statistic $\wh{D}_n$ for the first-order autoregressive process.
More precisely, we have established the almost sure convergence and the asymptotic normality for all three estimators $\wh{\theta}_n$, $\wh{\rho}_n$ and
$\wh{D}_n$. In addition, we have proposed a new bilateral statistical procedure for testing serial correlation, built on $\wh{D}_{n}$. All these results give a new light on the well-known test of Durbin-Watson in a context of lagged dependent random variables. From a practical standpoint and for a matter of completeness, we may wonder about the estimation of the true values of the parameters $\theta$ and $\rho$.
As soon as $ \theta \neq -\rho$, we can estimate $\theta$ and $\rho$ by solving the nonlinear system of two equations
given, for $\wh{a}_n = \wh{\theta}_n +\wh{\rho}_n$ and $\wh{b}_n = \wh{\rho}_n / \wh{\theta}_n$, by

\begin{equation*}
\left\{
\begin{array}{lcl}
{\displaystyle \lim_{n \rightarrow \infty} \wh{a}_n }
& = &  \theta + \rho \vspace{1ex}\\
{\displaystyle \lim_{n \rightarrow \infty}  \wh{b}_n }
& = & \theta \rho
\end{array}\right.
\hspace{1cm} \text{a.s.}
\end{equation*}
One can easily find two couples of solutions, symmetrically linked together. For example, assuming $\theta < \rho$, we propose to make use of
\begin{equation*}
\widetilde{\theta}_n = \frac{\wh{a}_n - \sqrt{\wh{a}_n^{\, 2} - 4 \wh{b}_n}}{2} 
\hspace{1cm} \text{and}\hspace{1cm} 
\widetilde{\rho}_n = \frac{\wh{a}_n + \sqrt{\wh{a}_n^{\, 2} - 4 \wh{b}_n}}{2},
\end{equation*}
merely inverting the values of $\wt{\theta}_n$ and $\wt{\rho}_n$ whether, for some statistical argument, we would rather choose $\theta > \rho$. By the same token,
it is also possible to estimate the true variances $\sigma^2$, $\sigma^2_{\theta}$, $\sigma^2_{\rho}$ and $\sigma^2_{D}$. For example, 
via convergence \eqref{ASCVGSIG}, we propose to estimate $\sigma^2$ by
\begin{equation*}
\widetilde{\sigma}_n^{\, 2}=
\left(\frac{ (1 - \wh{b}_{n})(1+\wh{b}_{n})^3  }{ (1 + \wh{b}_{n})^2 - (\wh{a}_{n}\wh{b}_{n})^2 }\right) \wh{\sigma}_{n}^2.
\end{equation*}
\ \vspace{1ex} \par
This work lifts the veil on a set of questions for long left unanswered about the pioneer works of Durbin and Watson
\cite{DurbinWatson50}, \cite{DurbinWatson51}, \cite{DurbinWatson71}, \cite{Malinvaud61}, \cite{NerloveWallis66}. 
As an improvement, it would be useful to extend our analysis to $p$-order autoregressive processes.
In addition, it would be very interesting to investigate the asymptotic behavior of the
Durbin-Watson statistic $\wh{D}_{n}$ in the explosive case where $\vert \theta \vert > 1$ or $\vert \rho \vert > 1$.


\section*{Appendix A}

\begin{center}
{\small PROOFS OF THE AUTOREGRESSIVE PARAMETER RESULTS}
\end{center}

\renewcommand{\thesection}{\Alph{section}} 
\renewcommand{\theequation}
{\thesection.\arabic{equation}} \setcounter{section}{1}  
\setcounter{equation}{0}


\subsection*{}
\begin{center}
{\bf A.1. Proof of Theorem \ref{THM_ASCVGTHETA}.}
\end{center}


We start with some useful technical lemmas we shall make repeatedly use of. The proof of the first one may 
be found in \cite{Duflo97} page 24.

\begin{lem}
\label{NOISE_LEMMA}
Assume that $(V_{n})$ is a sequence of independent and identically distributed random variables
such that, for some $a \geq 1$, $\dE[|V_1|^a]$ is finite. Then,
\begin{equation}
\label{CVGNOISE}
\lim_{n\rightarrow \infty} \frac{1}{n} \sum_{k=1}^{n} |V_k|^a=\dE[|V_1|^a] \hspace{1cm}\textnormal{a.s.}
\end{equation}
and
\begin{equation}
\label{SUPNOISE}
\sup_{1 \leq k \leq n}  |V_k|=o(n^{1/a}) \hspace{1cm}\textnormal{a.s.}
\end{equation}
\end{lem}

\begin{lem}
\label{X_LEMMA}
Assume that $(V_{n})$ is a sequence of independent and identically distributed random variables
such that, for some $a \geq 1$, $\dE[|V_1|^a]$ is finite. If $(X_n)$ satisfies \eqref{AR}
with $\vert \theta \vert < 1$, $\vert \rho \vert < 1$, then
\begin{equation}
\label{CVGX}
\sum_{k=1}^{n} |X_k|^a=O(n) \hspace{1cm}\textnormal{a.s.}
\end{equation}
and
\begin{equation}
\label{SUPX}
\sup_{1 \leq k \leq n}  |X_k|=o(n^{1/a}) \hspace{1cm}\textnormal{a.s.}
\end{equation}
\end{lem}

\begin{rem}
In the particular case $a=2$, we obtain that
$$ \sum_{k=1}^{n} X_k^2=O(n) 
\hspace{1cm}\text{and}\hspace{1cm}
\sup_{1 \leq k \leq n} X_k^2=o(n)
\hspace{1cm}\textnormal{a.s.}$$
\end{rem}

\begin{proof} It follows from \eqref{AR} that for all $n\geq 1$,
\begin{equation}
\label{INEGX}
 | X_{n} | \leq |\theta|^{n} |X_0| + \sum_{k=1}^{n} |\theta |^{n-k} |\veps_{k}|.
 \end{equation}
Consequently, as $|\theta |< 1$, we obtain that
\begin{equation}
\label{MAJX}
\sup_{1 \leq k \leq n} |X_k| \leq \frac{1}{1-|\theta |} \Bigl( |X_0| + \sup_{1 \leq k \leq n} | \veps_{k} | \Bigr).
\end{equation}
By the same token, as $|\rho |< 1$, we also deduce from \eqref{AR} that
\begin{equation}
\label{MAJNOISE}
\sup_{1 \leq k \leq n} |\veps_k| \leq \frac{1}{1-|\rho |} \Bigl( |\veps_0| + \sup_{1 \leq k \leq n} | V_{k} | \Bigr).
\end{equation}
Hence, \eqref{SUPNOISE} together with \eqref{MAJX} and \eqref{MAJNOISE} obviously imply 
\begin{equation*}
\sup_{1 \leq k \leq n}  |X_k|=o(n^{1/a}) \hspace{1cm}\text{a.s.}
\end{equation*}
Furthermore, let $b$ be the conjugate exponent of $a$,
$$
\frac{1}{a}+\frac{1}{b}=1.
$$
It follows from \eqref{INEGX} and Holder's inequality that for all $n\geq 1$,
\begin{equation*}
 | X_{n} | \leq  \left( |\theta|^{n} |X_0|^a + \sum_{k=1}^{n} |\theta |^{n-k} |\veps_{k}|^a\right)^{1/a}\left(  \sum_{k=0}^{n} |\theta |^{n-k} \right)^{1/b}
 \end{equation*}
 which implies that
 \begin{eqnarray*}
 | X_{n} |^a &\leq & \left(  \sum_{k=0}^{n} |\theta |^{n-k} \right)^{a/b} \left( |\theta|^{n} |X_0|^a + \sum_{k=0}^{n} |\theta |^{n-k} |\veps_{k}|^a\right), \\
 &\leq & \left(  \sum_{k=0}^{\infty} |\theta |^{k} \right)^{a/b} \left( |\theta|^{n} |X_0|^a + \sum_{k=1}^{n} |\theta |^{n-k} |\veps_{k}|^a\right), \\
 &\leq & \Bigl(  1- | \theta | \Bigr)^{-a/b} \left( |\theta|^{n} |X_0|^a + \sum_{k=1}^{n} |\theta |^{n-k} |\veps_{k}|^a\right).
 \end{eqnarray*}
 Consequently,
\begin{eqnarray}
 \sum_{k=1}^{n} | X_{k} |^a & \leq & \Bigl(  1- | \theta | \Bigr)^{-a/b} \left(  \sum_{k=1}^{n}|\theta|^{k} |X_0|^a +  
 \sum_{k=1}^{n}\sum_{\ell=1}^{k} |\theta |^{k-\ell} |\veps_{\ell}|^a\right),  \notag \\
 &\leq & \Bigl(  1- | \theta | \Bigr)^{-a/b} \left(  |X_0|^a \sum_{k=1}^{n}|\theta|^{k}  +  
 \sum_{\ell=1}^{n}  |\veps_{\ell}|^a \sum_{k=\ell}^{n} |\theta |^{k-\ell}\right),  \notag \\ 
 &\leq &  \Bigl(  1- | \theta | \Bigr)^{-a} \left(  |X_0|^a   +  
 \sum_{k=1}^{n}  |\veps_{k}|^a \right).
\label{MAJSX} 
 \end{eqnarray} 
Via the same lines, we also obtain that
 \begin{equation}
 \sum_{k=1}^{n} | \veps_{k} |^a \leq 
 \Bigl(  1- | \rho | \Bigr)^{-a} \left(  |\veps_0|^a   +  
 \sum_{k=1}^{n}  |V_{k}|^a \right).
\label{MAJSNOISE} 
 \end{equation}
Finally, \eqref{CVGNOISE} together with \eqref{MAJSX} and \eqref{MAJSNOISE} lead to \eqref{CVGX}, which completes the proof of 
Lemma \ref{X_LEMMA}.
\end{proof}

\noindent{\bf Proof of Theorem \ref{THM_ASCVGTHETA}.}
We easily deduce from \eqref{AR} that the process $(X_n)$ satisfies the fundamental autoregressive equation given, for all $n\geq 2$, by
\begin{equation}
\label{NEWAR}
X_{n}=(\theta + \rho) X_{n-1} - \theta\rho X_{n-2} + V_{n}.
\end{equation}
For all $n\geq 0$, let
\begin{eqnarray} 
\label{DEFSN} 
S_{n} &=&\sum_{k=0}^{n}X_k^2, \\
\label{DEFPN} 
P_{n} &=&\sum_{k=1}^{n}X_{k} X_{k-1},\\
\label{DEFMN} 
M_n &=& \sum_{k=1}^{n} X_{k-1} V_{k}
\end{eqnarray} 
where $P_0=0$ and $M_0=0$. It is not hard to see from \eqref{NEWAR} that for all $n\geq 2$,
$$
P_{n} = (\theta + \rho) S_{n-1} - \theta\rho P_{n-1} + M_{n} + \rho X_0 (\veps_0 - X_0)
$$
which implies that
\begin{equation}
(1+\theta\rho) P_{n}= (\theta + \rho) S_{n-1} +M_{n} + \theta\rho X_nX_{n-1} + \rho X_0 (\veps_0 - X_0).
\label{DECOPN}
\end{equation}
Via \eqref{THETA_EST}, \eqref{DECOPN} leads to the main decomposition
\begin{equation}
\label{DECOTHETA_EST}
\wh{\theta}_{n} = \frac{\theta + \rho}{1+\theta\rho} + \frac{1}{1+\theta\rho} \frac{M_{n}}{S_{n-1}} + \frac{1}{1+\theta\rho} \frac{R_{n}}{S_{n-1}}
\end{equation}
where the remainder term
$$
R_n =  \theta \rho X_nX_{n-1} +  \rho X_0 (\veps_0 - X_0).
$$
For all $n \geq 1$, denote by $\cF_{n}$ the $\sigma $-algebra of the events occurring up to time $n$, $\cF_{n}=\sigma(X_0, \veps_0, V_1, \hdots, V_n)$.
We infer from \eqref{DEFMN} that $(M_n)$ is a locally square-integrable real martingale \cite{Duflo97}, \cite{HallHeyde80} with predictable quadratic variation
given by $\langle M \rangle_0=0$ and for all $n \geq 1$,
\begin{eqnarray*}
\langle M \rangle_n&=&\sum_{k=1}^n \dE[(M_k - M_{k-1})^2|\cF_{k-1}] ,\\
&=&\sum_{k=1}^n \dE[X_{k-1}^2 V_{k}^2|\cF_{k-1}] = \sigma^2 S_{n-1}.
\end{eqnarray*}
Furthermore, it follows from \eqref{NEWAR}  and Corollary 1.3.25 of \cite{Duflo97}
that $n=O(S_{n})$ a.s. Then, we deduce from the strong law of large
numbers for martingales given e.g. by Theorem 1.3.15 of \cite{Duflo97} that
$$
\lim_{n\rightarrow \infty} \frac{M_{n}}{\langle M \rangle_{n}} = 0 \hspace{1cm} \text{a.s.}
$$
which of course ensures that
\begin{equation}
\label{CVGMN}
\lim_{n\rightarrow \infty} \frac{M_{n}}{S_{n-1}} = 0 \hspace{1cm} \text{a.s.}
\end{equation}
It remains to show that the remainder $R_n=o(S_{n-1})$ a.s. We have from \eqref{SUPX} with $a=2$ that $|X_n|=o(\sqrt{n})$ a.s.
which implies that  $R_n=o(n)$ a.s. However, we already saw that $n=O(S_{n})$ a.s. Hence,
\begin{equation}
\label{CVGRN}
\lim_{n\rightarrow \infty} \frac{R_{n}}{S_{n-1}} = 0 \hspace{1cm} \text{a.s.}
\end{equation}
Finally, it follows from \eqref{DECOTHETA_EST} together with \eqref{CVGMN} and \eqref{CVGRN} that
$$
\lim_{n\rightarrow \infty} \wh{\theta}_{n}= \frac{\theta + \rho}{1 + \theta \rho}
\hspace{1cm} \text{a.s.}
$$
which achieves the proof of 
Theorem \ref{THM_ASCVGTHETA}.
\hfill
$\mathbin{\vbox{\hrule\hbox{\vrule height1ex \kern.5em\vrule height1ex}\hrule}}$


\subsection*{}
\begin{center}
{\bf A.2. Proof of Theorem \ref{THM_CLTTHETA}.}
\end{center}


In order to establish the asymptotic normality of the least squares estimator $\wh{\theta}_{n}$, it is necessary
to be more precise in Lemma \ref{X_LEMMA} with $a=2$.

\begin{lem}
\label{CVGM2X_LEMMA}
Assume that the initial values $X_0$ and $\veps_0$ are square-integrable and that $(V_{n})$ is a sequence of square-integrable, independent 
and identically distributed random variables with zero mean and variance $\sigma^2 > 0$. Then,
\begin{equation}
\label{CVGM2X}
\lim_{n\rightarrow \infty} \frac{1}{n} \sum_{k=1}^{n} X_k^2=\ell \hspace{1cm}\textnormal{a.s.}
\end{equation}
where the limiting value
\begin{equation}
\label{DEFl}
\ell=\frac{\sigma^2 (1+\theta \rho)}{(1-\theta^2)(1-\theta \rho)(1-\rho^2)}.
\end{equation}
In addition, if $\ell_1 =  \theta^{*}\ell$, then
\begin{equation}
\label{CVGMXX}
\lim_{n\rightarrow \infty} \frac{1}{n} \sum_{k=1}^{n} X_k X_{k-1}=\ell_1  \hspace{1cm}\textnormal{a.s.}
\end{equation}
\end{lem}

\begin{proof}
We deduce from the fundamental autoregressive equation \eqref{NEWAR} together with straightforward calculations that
for all $n\geq 2$,
\begin{eqnarray*}
S_{n} &=& (\theta + \rho)^2 S_{n-1} + (\theta \rho)^2 S_{n-2} + L_n
- 2 \theta \rho (\theta + \rho) P_{n-1} \\
\vspace{2ex}
& & \hspace{1cm}+ 2 (\theta + \rho) M_{n} - 2 \theta \rho N_{n} + \xi_1
\end{eqnarray*}
where $S_n$, $P_n$ and $M_n$ are respectively given by \eqref{DEFSN}, \eqref{DEFPN} and \eqref{DEFMN}, 
the last term 
$\xi_1=(1-2\theta \rho - \rho^2)X_0^2 + \rho^2 \veps_0^2 + 2\theta \rho X_0\veps_0 + 2\rho (\veps_0 - X_0) V_1$
and, for all $n \geq 2$,
\begin{eqnarray} 
\label{DEFLN} 
L_{n} &=&\sum_{k=1}^{n}V_k^2,  \\
\label{DEFNN} 
N_n &=& \sum_{k=2}^{n} X_{k-2} V_{k}.
\end{eqnarray} 
Consequently, 
\begin{equation}
(1-(\theta + \rho)^2- (\theta \rho)^2 ) S_{n}= L_n - 2 \theta \rho (\theta + \rho) P_{n} + 2 (\theta + \rho) M_{n} - 2 \theta \rho N_{n} - T_n
\label{DECOSN}
\end{equation}
where the remainder term 
$$T_n=((\theta + \rho)^2 + (\theta \rho)^2) X_{n}^2 + (\theta \rho)^2 X_{n-1}^2 - 2 \theta \rho (\theta + \rho) X_{n} X_{n-1} - \xi_1. $$
It follows from \eqref{CVGNOISE} with $a=2$ that
\begin{equation}
\label{LLNNOISE}
\lim_{n\rightarrow \infty} \frac{L_n}{n} =\sigma^2 \hspace{1cm}\text{a.s.}
\end{equation}
In addition, we already saw from equation \eqref{DECOPN} that $P_n= \theta^{*}S_{n-1}+o(S_{n-1})$ a.s.  which clearly implies 
\begin{equation}
P_{n}= \theta^{*}S_{n}+o(S_{n}) \hspace{1cm}\text{a.s.}
\label{RESPN}
\end{equation}
Moreover, $(N_n)$ given by \eqref{DEFNN} is a locally square-integrable real martingale sharing the same almost
sure properties than $(M_n)$. More precisely, its predictable quadratic variation is
given by $\langle N \rangle_n=\sigma^2 S_{n-2}$ which means that
\begin{equation}
\lim_{n\rightarrow \infty} \frac{M_{n}}{S_{n}} = 0
\hspace{1cm}\text{and} \hspace{1cm}
\lim_{n\rightarrow \infty} \frac{N_{n}}{S_{n}} = 0
\hspace{1cm}\text{a.s.}
\label{RESMNNN}
\end{equation}
Furthermore, we have from \eqref{SUPX} with $a=2$ that $X_n^2=o(n)$ a.s.
It ensures by use of $n=O(S_n)$ a.s. that
\begin{equation}
\lim_{n\rightarrow \infty} \frac{T_{n}}{S_{n}} = 0
\hspace{1cm}\text{a.s.}
\label{RESTN}
\end{equation}
Therefore, it follows from the conjunction of \eqref{DECOSN}, \eqref{RESPN}, \eqref{RESMNNN}, and \eqref{RESTN} that
\begin{equation}
(1-(\theta + \rho)^2- (\theta \rho)^2 - 2 \theta \rho (\theta + \rho) \theta^{*} ) S_{n}= L_n + o(S_n) \hspace{1cm}\text{a.s.}
\label{DECOFINSN}
\end{equation}
Finally, dividing both sides of \eqref{DECOFINSN} by $n$ and letting $n$ goes to infinity, we deduce from \eqref{LLNNOISE} that
\begin{equation*}
\lim_{n\rightarrow \infty} 
\frac{S_{n}}{n} =  \frac{\sigma^2 (1+\theta \rho)}{(1-\theta \rho)(1-\theta^2)(1-\rho^2)} \hspace{1cm} \text{a.s.}
\end{equation*}
\begin{equation*}
\lim_{n\rightarrow \infty} 
\frac{P_{n}}{n}  =  \frac{\sigma^2 (\theta + \rho)}{(1-\theta \rho)(1-\theta^2)(1-\rho^2)} \hspace{1cm} \text{a.s.}
\end{equation*}
These two limits will often be used in all the sequel.
\end{proof}
\ \vspace{-1ex}\\
\noindent{\bf Proof of Theorem \ref{THM_CLTTHETA}.}
We are now in the position to prove the asymptotic normality of $\wh{\theta}_{n}$. We have from the main decomposition
\eqref{DECOTHETA_EST} that for all $n \geq 2$, 
\begin{equation}
\label{A1_THETA_TLC}
\sqrt{n} \left( \wh{\theta}_{n} - \theta^{*} \right) = \sqrt{n}\left(\frac{\sigma^2}{1 + \theta\rho} \right)\frac{M_{n}}{\langle M \rangle_{n}} +
\sqrt{n}\left(\frac{1}{1 + \theta\rho} \right) \frac{R_{n}}{S_{n-1}}.
\end{equation}
We shall make use of the central limit theorem for martingales given e.g. by Corollary 2.1.10 of  \cite{Duflo97},
 to establish the asymptotic normality of the first term in 
the right-hand side of \eqref{A1_THETA_TLC}. On the other hand, we will also show that the second term
$\sqrt{n}R_{n}/S_{n-1}$ goes to zero almost surely. First of all, it follows from \eqref{CVGM2X} that
\begin{equation}
\label{A1_THETA_LIM_INC}
\lim_{n\rightarrow \infty} 
\frac{\langle M \rangle_{n}}{n} = \sigma^2 \ell  \hspace{1cm} \text{a.s.}
\end{equation}
From now on, in order to apply the central limit theorem for martingales, it is necessary to prove that the Lindeberg condition
is satisfied. For all $n \geq 1$, denote $\Delta M_{n} = X_{n-1} V_{n}$.  One only has to show that for all $\veps > 0$,
\begin{equation}
\label{A1_THETA_DEF_LINDEBERG}
\frac{1}{n} \sum_{k=1}^{n} \dE \left[ | \Delta M_{k} |^2 \rI_{| \Delta M_{k} |  \geq \veps \sqrt{n}} | \cF_{k-1} \right] \limp 0.
\end{equation}
However, we have assumed that $(V_n)$ has a finite moment of order 4, $\tau^4=\dE[V_1^4]$. Hence,
for all $n \geq 1$, $\dE [ | \Delta M_{n} |^4 | \cF_{n-1} ]=\dE [ X_{n-1}^4 V_n^4 | \cF_{n-1} ]= \tau^4 X_{n-1}^4$.
In addition, we deduce from \eqref{CVGX} with $a=4$ that
\begin{equation}
\label{CVGX4}
 \vspace{-2ex}
\sum_{k=1}^{n} X_k^4=O(n) \hspace{1cm}\text{a.s.}
\end{equation}
Therefore, for all $\veps > 0$,
\begin{eqnarray*}
\frac{1}{n} \sum_{k=1}^{n} \dE \left[ | \Delta M_{k} |^2 \rI_{| \Delta M_{k} |  \geq \veps \sqrt{n}} | \cF_{k-1} \right] 
& \leq & \frac{1}{\veps^2n^2} \sum_{k=1}^{n} \dE \left[ | \Delta M_{k} |^4  | \cF_{k-1} \right],   \vspace{-1ex} \\
& \leq & \frac{\tau^4}{\veps^2n^2} \sum_{k=1}^{n} X_{k-1}^4.
\end{eqnarray*}
Consequently, \eqref{CVGX4} ensures that
$$
\frac{1}{n} \sum_{k=1}^{n} \dE \left[ | \Delta M_{k} |^2 \rI_{| \Delta M_{k} |  \geq \veps \sqrt{n}} | \cF_{k-1} \right]= 
O ( n^{-1} )  \hspace{1cm}\text{a.s.}
$$
and the Lindeberg condition is clearly satisfied. We can conclude from the central limit theorem for martingales
that
\begin{equation}
\label{A1_THETA_CLTMN}
\frac{1}{\sqrt{n}} M_{n} \liml \cN \left(0, \sigma^2 \ell \right)
\end{equation}
in which the asymptotic variance is the deterministic limit 
given by (\ref{A1_THETA_LIM_INC}). Moreover, as $\ell >0$, we have from 
\eqref{A1_THETA_CLTMN} and Slutsky's lemma that
\begin{equation}
\label{A1_THETA_CLTMNORM}
\sqrt{n} \frac{M_{n}}{\langle M \rangle_{n}}  \liml \cN \left(0, \sigma^{-2} \ell^{-1} \right).
\end{equation}
It only remains to prove that $\sqrt{n} R_n=o(S_{n-1})$ a.s. We have from \eqref{SUPX} with $a=4$ that $|X_n|=o(n^{1/4})$ a.s.
which implies that  $\sqrt{n} R_n=o(n)$ a.s. Hence,
\begin{equation}
\label{CVGCLTRN}
\lim_{n\rightarrow \infty} \sqrt{n} \frac{R_{n}}{S_{n-1}} = 0 \hspace{1cm} \text{a.s.}
\end{equation}
Finally, it follows from \eqref{A1_THETA_TLC} together with \eqref{A1_THETA_CLTMNORM} and \eqref{CVGCLTRN} that
$$
\sqrt{n} \left( \wh{\theta}_{n} - \theta^{*} \right) \liml
\cN \left(0, \sigma^2_{\theta}\right)
$$
where the asymptotic variance
$$
 \sigma^2_{\theta}=\frac{\sigma^2}{\ell (1+\theta \rho)^2} =\frac{(1-\theta^2)(1-\theta\rho)(1-\rho^2)}{(1+\theta\rho)^3}
$$
which achieves the proof of Theorem \ref{THM_CLTTHETA}.
\hfill
$\mathbin{\vbox{\hrule\hbox{\vrule height1ex \kern.5em\vrule height1ex}\hrule}}$


\subsection*{}
\begin{center}
{\bf A.3. Proof of Theorem \ref{THM_LILTHETA}.}
\end{center}


Denote by $f_{n}$ the explosion coefficient associated with the locally square-integrable real martingale 
$(M_{n})$, given for all $n \geq 0$, by
\begin{equation}
\label{DEFFN}
f_{n} = \frac{X_{n}^2}{S_n}.
\end{equation}
It clearly follows from \eqref{CVGM2X} that $f_n$ tends to zero almost surely. Consequently, by virtue of the quadratic strong law for martingales given by Theorem 3 of \cite{Bercu04} or \cite{BercuCenacFayolle09},
\begin{equation}
\label{A1_THETA_LFQ}
\lim_{n\rightarrow \infty} 
\frac{1}{\log n} \sum_{k=1}^{n} f_{k} \left( \frac{M_{k}^2}{S_{k-1}} \right) = \sigma^2 \hspace{1cm} \text{a.s.}
\end{equation}
In addition, by summation of equation \eqref{DECOTHETA_EST}, we have for all $n \geq 1$, 
\begin{eqnarray*}
\sum_{k=1}^{n} f_{k} S_{k-1} \left( \wh{\theta}_{k} - \theta^{*} \right)^2 & = & \frac{1}{(1+\theta \rho)^2} \sum_{k=1}^{n} f_{k} \left( \frac{M_{k}^2}{S_{k-1}} \right)
+ \frac{1}{(1 + \theta \rho)^2} \sum_{k=1}^{n} f_{k} \left( \frac{R_{k}^2}{S_{k-1}} \right) \\
& & \hspace{1cm} +  \frac{2}{(1 + \theta \rho)^2} \sum_{k=1}^{n} f_{k} \left( \frac{M_{k}R_{k}}{S_{k-1}} \right).
\end{eqnarray*}
We already saw from \eqref{CVGCLTRN} that $R_n^2=o(S_{n-1})$ a.s. Moreover, by the elementary inequality $x \leq - \log(1 -x)$  where $0\leq x \leq 1$, we obtain
that $f_n \leq - \log(1-f_n)$ which means that $f_n \leq \log S_n - \log S_{n-1}$. Thus,
\begin{equation*}
\sum_{k=1}^{n} f_{k} \left( \frac{R_{k}^2}{S_{k-1}} \right)= O(1) + o \left( \sum_{k=1}^{n} f_{k} \right) = O(1) + o \left( \log S_n \right)= o \left( \log n \right)\hspace{1cm} \text{a.s.}
\end{equation*}
Consequently, the second term of the summation is negligible compared to the first one. Furthermore,  the third one is a cross-term and this ensures that it also plays a negligible role compared to the first term. Thereby,
\begin{equation}
\label{A1_THETA_LFQ2}
\lim_{n\rightarrow \infty} 
\frac{1}{\log n} \sum_{k=1}^{n} f_{k} S_{k-1} \left( \wh{\theta}_{k} - \theta^{*} \right)^2 = \frac{\sigma^2}{(1+ \theta \rho)^2} \hspace{1cm} \text{a.s.}
\end{equation}
Finally, as in the proof of Corollary 8 in \cite{Bercu04}, we deduce from \eqref{CVGM2X} and \eqref{A1_THETA_LFQ2} that
\begin{eqnarray*}
\lim_{n\rightarrow \infty} 
\frac{1}{\log n} \sum_{k=1}^{n}  \left( \wh{\theta}_{k} - \theta^{*} \right)^2 &=& \frac{\sigma^2}{\ell(1+ \theta \rho)^2} \hspace{1cm} \text{a.s.}\\
&=& \frac{(1 - \theta^2)(1-\theta \rho)(1 - \rho^2)}{(1 + \theta \rho)^3} \hspace{1cm} \text{a.s.}
\end{eqnarray*}
which completes the proof of the quadratic strong law \eqref{QSLTHETA}. We shall now proceed to the proof of the law of iterated logarithm
given by \eqref{LILTHETA}. Kolmogorov's law of iterated logarithm was extended to the martingale framework by Stout \cite{Stout70}, \cite{Stout74}, and 
a simplified version of this result may be found in Corollary 6.4.25 of \cite{Duflo97}. In order to apply 
the law of iterated logarithm for martingales, it is only necessary to verify that
\begin{equation}
\label{A1_THETA_LLI_COND}
\sum_{k=1}^{+\infty} \frac{X_{k}^4}{k^2} < +\infty \hspace{1cm} \text{a.s.}
\end{equation}
For all $n \geq 0$, denote 
$$T_{n} = \sum_{k=1}^{n} X_{k}^4$$
with $T_0=0$. We clearly have
\begin{equation*}
\sum_{k=1}^{+\infty} \frac{X_{k}^4}{k^2} =  \sum_{k=1}^{+\infty} \frac{T_{k} - T_{k-1}}{k^2}
=\sum_{k=1}^{+\infty} \left(\frac{2k+1}{k^2(k+1)^2} \right) T_{k}.
\end{equation*}
However, we already saw from \eqref{CVGX4} that $T_n = O(n)$ a.s. Consequently, \\
$$
\sum_{k=1}^{+\infty} \frac{X_{k}^4}{k^2}=O\left( \sum_{k=1}^{+\infty} \frac{T_k}{k^3} \right)= O\left( \sum_{k=1}^{+\infty} \frac{1}{k^2} \right) = O(1)\hspace{1cm} \text{a.s.}
$$
which immediately implies (\ref{A1_THETA_LLI_COND}). Then, we obtain  from the law of iterated logarithm for martingales that
\begin{eqnarray*}  
\limsup_{n \rightarrow \infty} \left(\frac{\langle M \rangle_{n}}{2  \log \log \langle M \rangle_{n}} \right)^{1/2} \!\! \frac{M_n}{\langle M \rangle_{n}}
&=& - \liminf_{n \rightarrow \infty}
\left(\frac{\langle M \rangle_{n}}{2\log \log \langle M \rangle_{n}} \right)^{1/2} \!\! \frac{M_n}{\langle M \rangle_{n}} \vspace{2ex}\\
&=&1
\hspace{1cm}\text{a.s.}
\end{eqnarray*}
Whence, as $ \langle M \rangle_{n}= \sigma^2 S_{n-1}$, we deduce from \eqref{CVGM2X} that
\begin{eqnarray}  
\limsup_{n \rightarrow \infty} \left(\frac{n}{2  \log \log n} \right)^{1/2} \!\! \frac{M_n}{S_{n-1}}
&=& - \liminf_{n \rightarrow \infty}
\left(\frac{n}{2\log \log n} \right)^{1/2} \!\! \frac{M_n}{S_{n-1}} \notag \vspace{1ex}\\
&=&\frac{\sigma}{\sqrt{\ell}}
\hspace{1cm}\text{a.s.}
\label{LILMN}
\end{eqnarray}
Furthermore, we obviously have from \eqref{CVGCLTRN} that
\begin{equation}
\label{LILRN}
\vspace{1ex}
\lim_{n\rightarrow \infty} \left(\frac{n}{2\log \log n} \right)^{1/2} \!\! \frac{R_{n}}{S_{n-1}} = 0 \hspace{1cm} \text{a.s.}
\vspace{1ex}
\end{equation}
Finally, \eqref{LILTHETA} follows from the conjunction of \eqref{DECOTHETA_EST}, \eqref{LILMN} and \eqref{LILRN}, completing the proof
of Theorem \ref{THM_LILTHETA}.
\hfill
$\mathbin{\vbox{\hrule\hbox{\vrule height1ex \kern.5em\vrule height1ex}\hrule}}$
\newpage

\section*{Appendix B}

\begin{center}
{\small  PROOFS OF THE SERIAL CORRELATION PARAMETER RESULTS}
\end{center}

\renewcommand{\thesection}{\Alph{section}} 
\renewcommand{\theequation}
{\thesection.\arabic{equation}} \setcounter{section}{2}  
\setcounter{equation}{0}
\setcounter{lem}{0}


\subsection*{}
\begin{center}
{\bf B.1. Proof of Theorem \ref{THM_ASCVGRHO}.}
\end{center}


In order to establish the almost sure convergence of the least squares estimator $\wh{\rho}_{n}$,
it is necessary to start with a useful technical lemma.

\begin{lem}
\label{CVGM2XY_LEMMA}
Assume that the initial values $X_0$ and $\veps_0$ are square-integrable and that $(V_{n})$ is a sequence of square-integrable, independent 
and identically distributed random variables with zero mean and variance $\sigma^2 > 0$. Then,
\begin{equation}
\label{CVGMXY}
\lim_{n\rightarrow \infty} \frac{1}{n} \sum_{k=2}^{n} X_k X_{k-2}=\ell_2  \hspace{1cm}\textnormal{a.s.}
\end{equation}
where the limiting value
\begin{equation}
\label{DEFl2}
\ell_2=\frac{\sigma^2 \big( (\theta + \rho)^2 - \theta\rho (1+\theta\rho) \big) }{(1-\theta^2)(1-\theta \rho)(1-\rho^2)}.
\end{equation}
\end{lem}

\begin{proof} Proceeding as in the proof of Lemma \ref{CVGM2X_LEMMA}, we deduce from \eqref{NEWAR} that
for all $n\geq 2$,
\begin{equation}
\label{DEFQN}
Q_n=\sum_{k=2}^n X_k X_{k-2}=(\theta + \rho)P_{n-1} - \theta \rho S_{n-2} + N_n
\end{equation}
where $S_n$, $P_n$ and $N_n$ are respectively given by \eqref{DEFSN},  \eqref{DEFPN} and \eqref{DEFNN}.
We already saw in Appendix A that $N_n=o(n)$ a.s. Hence, it follows from \eqref{CVGM2X} and \eqref{CVGMXX} that
\begin{equation*}
\lim_{n\rightarrow \infty} \frac{Q_n}{n} = (\theta + \rho)\ell_1 - \theta \rho \ell = \ell_2 \hspace{1cm}\text{a.s.} \\
\end{equation*}
which achieves the proof of Lemma \ref{CVGM2XY_LEMMA}.
\end{proof}

\noindent{\bf Proof of Theorem \ref{THM_ASCVGRHO}.}
We are now in the position to prove the almost sure convergence of $\wh{\rho}_{n}$ to $\rho^{*}$ given by
\eqref{rhostar}. For all $n \geq 1$, denote
\begin{equation*} 
I_{n} =\sum_{k=1}^{n}\wh{\veps}_{k} \wh{\veps}_{k-1}
\hspace{1cm}\text{and}\hspace{1cm}
J_{n} =\sum_{k=0}^{n}\wh{\veps}_{k}^{\, \, 2}.
\end{equation*} 
It is not hard to see that
\begin{eqnarray}
\label{DECOIN}
I_{n} & = & P_n - \wh{\theta}_{n}(S_{n-1}+Q_n) + \wh{\theta}_{n}^{\, 2} P_{n-1}, \\
J_{n} & = & S_n - 2 \wh{\theta}_{n} P_n+\wh{\theta}_{n}^{\, 2} S_{n-1}.
\label{DECOJN}
\end{eqnarray} 
Consequently, it follows from convergence \eqref{ASCVGTHETA} together with \eqref{CVGM2X}, \eqref{CVGMXX} and \eqref{CVGMXY} that
\begin{eqnarray}
\lim_{n\rightarrow \infty} \frac{I_n}{n} & = & \ell_1 - \theta^{*} ( \ell + \ell_2)  + (\theta^{*})^2 \ell_1  
\hspace{1cm}\text{a.s.} \notag \\
& = & \theta \rho \theta^{*}  \ell (1 - (\theta^{*})^2 ) 
\hspace{1cm}\text{a.s.} \notag \\
& = &  \rho^{*}  \ell (1 - (\theta^{*})^2 ) 
\hspace{1cm}\text{a.s.}
\label{CVGIN}
\end{eqnarray}
since $\rho^{*}=\theta \rho \theta^{*}$. By the same token,
\begin{eqnarray}
\lim_{n\rightarrow \infty} \frac{J_n}{n} & = & \ell - 2 \theta^{*} \ell_1+ (\theta^{*})^2 \ell  
\hspace{1cm}\text{a.s.} \notag \\
& = &   \ell (1 - (\theta^{*})^2 ) 
\hspace{1cm}\text{a.s.}
\label{CVGJN}
\end{eqnarray}
One can observe that $\ell >0$ and $|\theta^{*}| <1$, which implies that $\ell (1 - (\theta^{*})^2 ) >0$.
Therefore, we deduce from \eqref{RHO_EST}, \eqref{CVGIN} and  \eqref{CVGJN} that
$$
\lim_{n\rightarrow \infty} \wh{\rho}_{n}  = \lim_{n\rightarrow \infty} \frac{I_n}{J_{n-1}} = \frac{\rho^{*}  \ell (1 - (\theta^{*})^2 )}{\ell (1 - (\theta^{*})^2 ) }
= \rho^{*}
\hspace{1cm}\text{a.s.}
$$
which completes the proof of  Theorem \ref{THM_ASCVGRHO}.
\hfill
$\mathbin{\vbox{\hrule\hbox{\vrule height1ex \kern.5em\vrule height1ex}\hrule}}$


\subsection*{}
\begin{center}
{\bf B.2. Proof of Theorem \ref{THM_CLTRHO}.}
\end{center}

First of all, we already saw from \eqref{DECOTHETA_EST} that
\begin{equation}
\label{DECOTHETACLT}
S_{n-1}\big(\wh{\theta}_{n} - \theta^{*} \big)= \frac{M_n}{1+\theta\rho}   + \frac{R_{n}(\theta)}{1+\theta\rho}
\end{equation}
where $R_n(\theta)=  \theta \rho X_nX_{n-1} +  \rho X_0 (\veps_0 - X_0)$.
Our goal is to find a similar decomposition for $\wh{\rho}_{n} - \rho^{*}$. On the one hand, we deduce from
\eqref{DECOPN} that
\begin{equation}
P_{n}= \theta^{*} S_{n} +\frac{M_{n}}{1+\theta\rho}  +  \frac{\xi_n^P}{1+\theta\rho} 
\label{DECOPNCLT}
\end{equation}
where $\xi_n^P=R_n(\theta)- (\theta + \rho)X_n^2$. On the other hand, we obtain from  \eqref{DEFQN} and \eqref{DECOPNCLT} that
\begin{equation}
Q_{n}= \big( (\theta + \rho)\theta^{*} - \theta \rho) \big)  S_{n} +\theta^{*} M_{n} + N_n +  \xi_n^Q
\label{DECOQNCLT}
\end{equation}
with $\xi_n^Q=\theta^{*} \xi_n^P-(\theta + \rho) X_nX_{n-1} +\theta\rho(X_n^2+X_{n-1}^2)$. 
Then, it follows from \eqref{DECOIN}, \eqref{DECOPNCLT} and \eqref{DECOQNCLT} together with tedious but straightforward calculations
that
\begin{equation}
I_{n}= \theta^{*}\big( \theta \rho -\theta^{*} \rho^{*} \big)  S_{n} +\left(\frac{1-\theta^{*} \rho^{*}}{1+\theta \rho}\right) M_{n} - \theta^{*}N_n 
-\big(\wh{\theta}_{n} - \theta^{*} \big)F_n + \xi_n^I
\label{DECOINCLT}
\end{equation}
where $F_n= S_n + Q_n - \big(\wh{\theta}_{n} + \theta^{*} \big)P_n$, 
$$
\xi_n^I=\wh{\theta}_{n} X_n^2-\wh{\theta}_{n}^{\, 2} X_nX_{n-1}+ \left(\frac{1+(\theta^{*})^2}{1+\theta \rho}\right)\xi_n^P - \theta^{*}\xi_n^Q.
$$
Via the same lines, we also find from \eqref{DECOJN}, \eqref{DECOPNCLT} and \eqref{DECOQNCLT} 
that
\begin{equation}
J_{n-1}= \big(1-(\theta^{*})^2 \big)  S_{n} -\left(\frac{2\theta^{*}}{1+\theta \rho}\right) M_{n} 
-\big(\wh{\theta}_{n} - \theta^{*} \big)G_n + \xi_n^J
\label{DECOJNCLT}
\end{equation}
where $G_n= 2P_n - \big(\wh{\theta}_{n} + \theta^{*} \big)S_n$, 
$$
\xi_n^J=-X_n^2 + 2\wh{\theta}_{n} X_nX_{n-1} -\wh{\theta}_{n}^{\, 2}( X_n^2+X_{n-1}^2) - \left(\frac{2\theta^{*}}{1+\theta \rho}\right)\xi_n^P.
$$
Replacing $I_{n}$ and $J_{n-1}$ by the expansions \eqref{DECOINCLT} and \eqref{DECOJNCLT}, we obtain from the
identity $J_{n-1}(\wh{\rho}_{n} - \rho^{*})=I_{n} - \rho^{*}J_{n-1}$ that
\begin{equation*}
J_{n-1}\big(\wh{\rho}_{n} - \rho^{*}\big)  = \left(\frac{1+\theta^{*} \rho^{*}}{1+\theta \rho}\right) M_{n} - \theta^{*}N_n -\big(\wh{\theta}_{n} - \theta^{*} \big)H_n + 
\xi_n^I - \rho^{*}\xi_n^J
\end{equation*}
where $H_n=F_n - \rho^{*}G_n$. One can observe that the leading term depending on $S_n$ vanishes as it should, since
$$
\left(\theta^{*}\big( \theta \rho -\theta^{*} \rho^{*}) \big)  - \rho^{*} \big(1-(\theta^{*})^2 \big) \right) =0.
$$
Consequently, we deduce from \eqref{DECOTHETACLT} that
\begin{equation}
\label{DECORHOCLT}
J_{n-1} \big(\wh{\rho}_{n} - \rho^{*} \big)= 
\frac{T_nM_n}{1+\theta\rho}  - \theta^{*}N_n +\frac{R_n(\rho)}{1+\theta\rho}
\end{equation}
where
\begin{eqnarray*}
T_n &=& 1+\theta^{*} \rho^{*} - \frac{H_n}{S_{n-1}},\\
R_n(\rho) &=& (1+\theta\rho)(\xi_n^I - \rho^{*}\xi_n^J)- \frac{R_n(\theta)H_n}{S_{n-1}}.
\end{eqnarray*}
In contrast to \eqref{DECOTHETACLT}, it was much more tricky to establish relation \eqref{DECORHOCLT}.
We are now in the position to prove the joint asymptotic normality of $\wh{\theta}_{n}$ and $\wh{\rho}_{n}$. 
Using the same approach as in \cite{WeiWinnicki90}, it follows from \eqref{DECOTHETACLT} and  \eqref{DECORHOCLT} that
\begin{equation}  
\label{DECOCLTTHETARHO}
\sqrt{n} \begin{pmatrix}
\ \wh{\theta}_{n} - \theta^{*}  \\
\ \wh{\rho}_{n} - \rho^{*}
\end{pmatrix}
 =\frac{1}{\sqrt{n}} A_n Z_n + B_n
\end{equation}
where
\begin{equation*}
Z_n= 
\begin{pmatrix}
M_n  \\
N_n
\end{pmatrix},
\vspace{2ex}
\end{equation*}
\begin{equation*}
A_n= \frac{n}{1+ \theta \rho}
\begin{pmatrix}
\displaystyle{\frac{1}{S_{n-1}}} & \ \ 0  \vspace{1ex}\\
\displaystyle{\frac{T_n}{J_{n-1}}} & -\displaystyle{\frac{(\theta + \rho)}{J_{n-1}}}
\end{pmatrix}
\hspace{1cm}\text{and}\hspace{1cm}
B_n= \frac{\sqrt{n}}{1+ \theta \rho}
\begin{pmatrix}
\displaystyle{
\frac{R_n(\theta)}{S_{n-1}}} \vspace{1ex}\\
\displaystyle{
\frac{R_n(\rho)}{J_{n-1}}
}
\end{pmatrix}.
\end{equation*}
On the one hand, we obtain from \eqref{CVGM2X}, \eqref{CVGMXX}, \eqref{CVGMXY} and \eqref{CVGJN} that
\begin{equation}
\label{CVGAN}
\lim_{n\rightarrow \infty} A_n= A
\hspace{1cm}\text{a.s.}
\end{equation}
where $A$ is the limiting matrix given by
\begin{equation}
\label{DEFA}
A=\frac{1}{\ell  (1+ \theta \rho)(1 - (\theta^{*})^2 )}
\begin{pmatrix}
1 - (\theta^{*})^2 & \ \ 0  \vspace{1ex}\\
\theta \rho + (\theta^{*})^2 & -(\theta + \rho)
\end{pmatrix}.
\end{equation}
On the other hand, as in the proof of \eqref{CVGCLTRN}, we clearly have
\begin{equation}
\label{CVGBN}
\lim_{n\rightarrow \infty} B_n= \begin{pmatrix}
\ 0\  \vspace{1ex}\\
\ 0\
\end{pmatrix}\hspace{1cm} \text{a.s.}
\end{equation}
Furthermore, $(Z_n)$ is a two-dimensional real martingale \cite{Duflo97}, \cite{HallHeyde80} 
with increasing process given, for all $n \geq 2$, by
\begin{equation*}
\langle Z \rangle_n=
\sigma^2 \begin{pmatrix}
S_{n-1} & P_{n-1} \\
P_{n-1} & S_{n-2}
\end{pmatrix}.
\end{equation*}
We deduce from \eqref{CVGM2X} and \eqref{CVGMXX} that
\begin{equation}
\label{CVGIPZN}
\lim_{n\rightarrow \infty} \frac{1}{n} \langle Z \rangle_n= L
\hspace{1cm}\text{a.s.}
\end{equation}
where $L$ is the positive-definite symmetric matrix given by
\begin{equation}
\label{DEFL}
L=\sigma^2 \ell
\begin{pmatrix}
1 & \theta^{*}  \vspace{1ex}\\
\theta^{*} & 1
\end{pmatrix}.
\end{equation}
We also immediately derive from \eqref{CVGX4} that $(Z_n)$ satisfies the Lindeberg condition.
Therefore, we can conclude from the central limit theorem for multidimensional martingales 
given e.g. by Corollary 2.1.10 of \cite{Duflo97} that
\begin{equation}
\label{CLTZN}
\frac{1}{\sqrt{n}} Z_{n} \liml \cN \left(0, L \right).
\end{equation}
Finally, we find from the conjunction of \eqref{DECOCLTTHETARHO}, \eqref{CVGAN}, \eqref{CVGBN}, and \eqref{CLTZN} together with
Slutsky's lemma that
\begin{equation*}  
\sqrt{n} \begin{pmatrix}
\ \wh{\theta}_{n} - \theta^{*}  \\
\ \wh{\rho}_{n} - \rho^{*}
\end{pmatrix}
 \liml  \cN \Bigl(0, A L A^{\prime}\Bigr).
\end{equation*}
One can easily check the identity $\Gamma= A L A^{\prime}$ via \eqref{DEFA} and \eqref{DEFL}, 
where $\Gamma$ is given by \eqref{DEFGAMMA}, which achieves 
the proof of  Theorem \ref{THM_CLTRHO}.
\hfill
$\mathbin{\vbox{\hrule\hbox{\vrule height1ex \kern.5em\vrule height1ex}\hrule}}$


\subsection*{}
\begin{center}
{\bf B.3. Proof of Theorem \ref{THM_LILRHO}.}
\end{center}


The proof of the quadratic strong law for $\wh{\theta}_{n}$ relies on the quadratic strong law for the martingale $(M_n)$ 
given by \eqref{A1_THETA_LFQ}
\begin{equation*}
\lim_{n\rightarrow \infty} 
\frac{1}{\log n} \sum_{k=1}^{n} f_{k} \left( \frac{M_{k}^2}{ S_{k-1}} \right) = \sigma^2 \hspace{1cm} \text{a.s.}
\end{equation*}
which implies that
\begin{equation}
\label{QSLMNFIN}
\lim_{n\rightarrow \infty} 
\frac{1}{\log n} \sum_{k=1}^{n}  \left( \frac{M_{k}}{S_{k-1}} \right)^2 = \frac{\sigma^2}{\ell} \hspace{1cm} \text{a.s.}
\end{equation}
In order to establish a similar result for $\wh{\rho}_{n}$, we shall introduce a suitable martingale $(L_n)$ which is
a linear combination of $(M_n)$ and $(N_n)$. The sequence $(L_n)$ is defined by $L_0=0$, $L_1=X_0V_1$ and, 
for all $n \geq 2$,
\begin{equation}
\label{DEFMLN}
L_n=M_n -aN_n = L_{1}+ \sum_{k=2}^n (X_{k-1} -a X_{k-2})V_k
\end{equation}
where 
$$
a=\frac{\theta + \rho}{\theta \rho + (\theta^{*})^2}.
$$
We infer from \eqref{DEFMN} and \eqref{DEFNN} together with \eqref{DEFMLN}
that $(L_n)$ is a locally square-integrable real martingale with predictable quadratic variation
given by $\langle L \rangle_0=0$,  $\langle L \rangle_1=\sigma^2 X_0^2$ and for all $n \geq 2$,
\begin{equation*}
\langle L \rangle_n = \sigma^2 \big( S_{n-1} -2a P_{n-1} + a^2 S_{n-2} \big).
\end{equation*}
Moreover, we clearly deduce from \eqref{CVGM2X} and \eqref{CVGMXX}  that
\begin{equation}
\label{LIMMLN}
\lim_{n\rightarrow \infty} 
\frac{\langle L \rangle_{n}}{n} = \sigma^2 \ell b  \hspace{1cm} \text{a.s.}
\end{equation}
where $b=1-2 a \theta^{*} + a^2$. It also comes from a tedious calculation that
\begin{equation}
\label{DEVb}
b=\frac{a^2(1-\theta^2)(1-\rho^2)c}{(\theta+\rho)^2(1+\theta \rho)^4}
\end{equation}
where $c=(\theta + \rho)^2(1+\theta \rho)^2 + (\theta \rho)^2(1-\theta^2)(1-\rho^2)$.
Then,  via the same arguments as in the proof of \eqref{QSLMNFIN}, we obtain
from \eqref{LIMMLN} that
\begin{equation}
\label{QSLMLNFIN}
\lim_{n\rightarrow \infty} 
\frac{1}{\log n} \sum_{k=1}^{n}  \left( \frac{L_{k}}{S_{k-1}}\right)^2 = \frac{\sigma^2 b}{ \ell } \hspace{1cm} \text{a.s.}
\end{equation}
Furthermore, it follows from \eqref{DECORHOCLT} that
\begin{equation}
\label{DECORHOQSL}
J_{n-1} \big(\wh{\rho}_{n} - \rho^{*} \big)= 
\left(\frac{\theta \rho +(\theta^{*})^2}{1+\theta\rho}\right)L_n  + \zeta_n=\frac{\theta^{*}L_n}{a}  + \zeta_n
\end{equation}
where
$$
\zeta_n=\left(\frac{T_n - \theta \rho -(\theta^{*})^2}{1+\theta\rho}\right)M_n +\frac{R_n(\rho)}{1+\theta\rho}.
$$
We obtain from \eqref{QSLMNFIN} and the almost sure convergence of $T_n$ to $\theta \rho +(\theta^{*})^2$ that
$$
\sum_{k=1}^{n}  \left( \frac{\zeta_{k}}{S_{k-1}}\right)^2 = o(\log n) \hspace{1cm} \text{a.s.}
$$
Consequently, \eqref{QSLMLNFIN} and \eqref{DECORHOQSL} lead to 
\begin{equation*}
\lim_{n\rightarrow \infty} 
\frac{1}{\log n} \sum_{k=1}^{n}  \left( \frac{J_{k-1}}{S_{k-1}}\right)^2 
\Big(\wh{\rho}_{k} - \rho^{*} \Big)^2= \frac{\sigma^2 b (\theta^{*})^2}{ a^2\ell } \hspace{1cm} \text{a.s.}
\end{equation*}
In addition, we get from \eqref{CVGJN} that
\begin{equation*}
\label{LIMJNSN}
\lim_{n\rightarrow \infty} 
\frac{J_n}{S_n} = 1- (\theta^{*})^{2} \hspace{1cm} \text{a.s.}
\end{equation*}
which implies that
\begin{equation}
\label{QSLRHOFIN}
\lim_{n\rightarrow \infty} 
\frac{1}{\log n} \sum_{k=1}^{n}  
\Big( \wh{\rho}_{k} - \rho^{*} \Big)^2= \frac{\sigma^2 b (\theta^{*})^2}{ a^2 \ell (1- (\theta^{*})^{2})^{2}} \hspace{1cm} \text{a.s.}
\end{equation}
However, we clearly have from \eqref{DEFl} that
$$
\frac{\sigma^2  (\theta^{*})^2}{ \ell (1- (\theta^{*})^{2})^{2}}=\frac{(\theta+\rho)^2(1-\theta \rho)(1+\theta\rho)}{(1-\theta^2)(1-\rho^2)}.
$$
Finally, we can deduce from \eqref{DEVb} and \eqref{QSLRHOFIN} that
\begin{equation*}
\lim_{n\rightarrow \infty} 
\frac{1}{\log n} \sum_{k=1}^{n}  
\Big( \wh{\rho}_{k} - \rho^{*} \Big)^2
=\frac{(1-\theta \rho)c}{(1+\theta \rho)^3}=\sigma^2_\rho \hspace{1cm} \text{a.s.}
\end{equation*}
which completes the proof of the quadratic strong law \eqref{QSLRHO}. The law of iterated logarithm given by \eqref{LILRHO}
is much more easy to handle. In order to make use of the law of iterated logarithm for the martingale $(L_n)$,
it is only necessary to verify that
\begin{equation*}
\sum_{k=1}^{+\infty} \frac{(X_{k}-aX_{k-1})^4}{k^2} < +\infty \hspace{1cm} \text{a.s.}
\end{equation*}
which of course follows from \eqref{A1_THETA_LLI_COND}. Consequently, we obtain that
\begin{eqnarray*}  
\limsup_{n \rightarrow \infty} \left(\frac{\langle L \rangle_{n}}{2  \log \log \langle L \rangle_{n}} \right)^{1/2} \!\! \frac{L_n}{\langle L \rangle_{n}}
&=& - \liminf_{n \rightarrow \infty}
\left(\frac{\langle L \rangle_{n}}{2\log \log \langle L \rangle_{n}} \right)^{1/2} \!\! \frac{L_n}{\langle L \rangle_{n}} \\
&=&1
\hspace{1cm}\text{a.s.}
\end{eqnarray*}
Therefore, we deduce from \eqref{LIMMLN} that
\begin{eqnarray*}  
\limsup_{n \rightarrow \infty} \left(\frac{n}{2  \log \log n} \right)^{1/2} \!\! \frac{L_n}{\langle L \rangle_{n}}
&=& - \liminf_{n \rightarrow \infty}
\left(\frac{n}{2\log \log n} \right)^{1/2} \!\! \frac{L_n}{\langle L \rangle_{n}} \notag \\
&=&\frac{1}{\sigma\sqrt{\ell b}}
\hspace{1cm}\text{a.s.}
\end{eqnarray*}
Whence, the convergence
\begin{equation*}
\lim_{n\rightarrow \infty} 
\frac{J_{n-1}}{\langle L \rangle_{n}} = \frac{1- (\theta^{*})^{2}}{\sigma^2 b} \hspace{1cm} \text{a.s.}
\end{equation*}
implies that
\begin{eqnarray}  
\limsup_{n \rightarrow \infty} \left(\frac{n}{2  \log \log n} \right)^{1/2} \!\! \frac{L_n}{J_{n-1}}
&=& - \liminf_{n \rightarrow \infty}
\left(\frac{n}{2\log \log n} \right)^{1/2} \!\! \frac{L_n}{J_{n-1}} \notag 
\vspace{-1ex}\\
&=&\frac{\sigma \sqrt{b}}{\sqrt{\ell}(1- (\theta^{*})^{2})}
\hspace{1cm}\text{a.s.}
\label{LILLN}
\end{eqnarray}
One can be convinced that the remainder term $\zeta_n$ at the right-hand side of \eqref{DECORHOQSL}
plays a negligible role compared to $L_n$. Finally, \eqref{DECORHOQSL} and \eqref{LILLN}
ensure that 
\begin{eqnarray*}  
\limsup_{n \rightarrow \infty} \left(\frac{n}{2  \log \log n} \right)^{1/2} \!\! \Big(\wh{\rho}_{n}-\rho^{*} \Big)
&=& - \liminf_{n \rightarrow \infty}
\left(\frac{n}{2\log \log n} \right)^{1/2} \!\!  \Big(\wh{\rho}_{n}-\rho^{*} \Big) \notag 
\vspace{-1ex}\\
&=&\frac{\sigma \sqrt{b} \theta^{*}}{a\sqrt{\ell}(1- (\theta^{*})^{2})}=\sigma_\rho
\hspace{1cm}\text{a.s.}
\end{eqnarray*}
which ends the proof of Theorem \ref{THM_LILRHO}.
\hfill
$\mathbin{\vbox{\hrule\hbox{\vrule height1ex \kern.5em\vrule height1ex}\hrule}}$


\section*{Appendix C}

\begin{center}
{\small PROOFS OF THE DURBIN-WATSON STATISTIC RESULTS}
\end{center}

\renewcommand{\thesection}{\Alph{section}} 
\renewcommand{\theequation}
{\thesection.\arabic{equation}} \setcounter{section}{3}  
\setcounter{equation}{0}


\subsection*{}
\begin{center}
{\bf C.1. Proof of Theorem \ref{THM_ASCVGDW}.}
\end{center}


First of all, we establish a very useful linear relation between the Durbin-Watson statistic $\wh{D}_n$
and the least squares estimator $\wh{\rho}_{n}$, which allows us to deduce the asymptotic behavior of $\wh{D}_n$.
For all $n \geq 1$, set
\begin{equation*} 
I_{n} =\sum_{k=1}^{n}\wh{\veps}_{k} \wh{\veps}_{k-1},
\hspace{1cm}
J_{n} =\sum_{k=0}^{n}\wh{\veps}_{k}^{\, \, 2},
\hspace{1cm}
K_{n} =\sum_{k=1}^{n}\big( \wh{\veps}_{k} - \wh{\veps}_{k-1}\big)^{2},
\end{equation*} 
\begin{equation*} 
f_{n} =\frac{\wh{\veps}_{n}^{\, \, 2}} {J_n}.
\end{equation*} 
It is not hard to see that 
\begin{equation*}
K_n=\sum_{k=1}^{n} \wh{\veps}_{k}^{\, \, 2} - 2 \sum_{k=1}^{n} \wh{\veps}_{k} \wh{\veps}_{k-1} + \sum_{k=1}^{n} \wh{\veps}_{k-1}^{\, \, 2}=
2\big(J_{n-1}-I_n\big)+ \wh{\veps}_{n}^{\, \, 2} -  \wh{\veps}_{0}^{\, \, 2}. 
\end{equation*}
Consequently, it follows from \eqref{DW_STAT} that 
\begin{equation}
\label{DEVDW}
\big(J_{n-1}+ \wh{\veps}_{n}^{\, \, 2}\big) \wh{D}_n= 2\big(J_{n-1}-I_n\big)+ \wh{\veps}_{n}^{\, \, 2} -  \wh{\veps}_{0}^{\, \, 2}.
\end{equation}
Therefore, dividing both sides of \eqref{DEVDW} by $J_{n-1}$, 
we obtain that
\begin{equation}
\label{DECODW}
\wh{D}_n= 2(1-f_n)\big(1-\wh{\rho}_n\big)+ \xi_n
\end{equation}
where 
$$
\xi_n =\frac{\wh{\veps}_{n}^{\, \, 2} - \wh{\veps}_{0}^{\, \, 2}}{J_{n}}.
$$
We already saw from \eqref{CVGJN} that
\begin{equation}
\label{NCVGJN}
\lim_{n\rightarrow \infty} \frac{J_n}{n} =  \ell (1 - (\theta^{*})^2 )
\hspace{1cm}\text{a.s.} 
\end{equation}
with $\ell (1 - (\theta^{*})^2 )>0$, which implies that both $f_n$ and $\xi_n$ converge to zero almost surely. Hence, we deduce from
\eqref{DECODW} that 
$$
\lim_{n\rightarrow \infty} \wh{D}_n =  D^{*}
\hspace{1cm}\text{a.s.} 
$$
where $D^{*}= 2( 1 - \rho^{*})$, which completes the proof of  Theorem \ref{THM_ASCVGDW}.
\hfill
$\mathbin{\vbox{\hrule\hbox{\vrule height1ex \kern.5em\vrule height1ex}\hrule}}$


\subsection*{}
\begin{center}
{\bf C.2. Proof of Theorem \ref{THM_CLTDW}.}
\end{center}


We shall now prove the asymptotic normality of $\wh{D}_{n}$ using \eqref{CLTRHO}. On the one hand, 
we clearly have from \eqref{DECODW} that 
\begin{equation}
\label{DECOCLTDW}
\sqrt{n} \big( \wh{D}_{n} - D^{*} \big) = -2 (1- f_n) \sqrt{n} \big( \wh{\rho}_{n} - \rho^{*} \big) + 2 (\rho^{*} - 1)\sqrt{n}f_n + \sqrt{n} \xi_n.
\end{equation}
On the other hand, we deduce from \eqref{SUPX} with $a=4$ that
\begin{equation*}
\sup_{1 \leq k \leq n}  \wh{\veps}_{k}^{\, \, 2}=o(\sqrt{n}) \hspace{1cm}\text{a.s.}
\end{equation*}
which, via \eqref{NCVGJN}, implies that
\begin{equation}
\label{CVGFN}
\lim_{n\rightarrow \infty} \sqrt{n}f_n =  0
\hspace{1cm}
\text{and}
\hspace{1cm}
\lim_{n\rightarrow \infty} \sqrt{n}\xi_n =  0
\hspace{1cm}\text{a.s.} 
\end{equation}
Then, it follows from  \eqref{CLTRHO}, \eqref{DECOCLTDW} and \eqref{CVGFN} together with
Slutsky's lemma that
\begin{equation*}
\sqrt{n} \left( \wh{D}_{n} - D^{*} \right) \liml \cN ( 0, \sigma^2_D)
\end{equation*}
where $\sigma^2_D=4 \sigma^2_\rho$, which achieves the proof of Theorem \ref{THM_CLTDW}.
\hfill
$\mathbin{\vbox{\hrule\hbox{\vrule height1ex \kern.5em\vrule height1ex}\hrule}}$
\subsection*{}
\begin{center}
{\bf C.3. Proof of Theorem \ref{THM_LILDW}.}
\end{center}

We immediately deduce from relation \eqref{DECODW} that
\begin{equation}
\label{DECOQSLDW1}
 \wh{D}_{n} - D^{*} = -2 (1- f_n) \big( \wh{\rho}_{n} - \rho^{*} \big) + \zeta_n 
\end{equation}
where $\zeta_n = 2 (\rho^{*} - 1)f_n + \xi_n$.
Consequently, by summation of \eqref{DECOQSLDW1}, we obtain that for all $n\geq 1$,
\begin{equation}
\label{DECOQSLDW2}
\sum_{k=1}^{n} \left( \wh{D}_{k} - D^{*} \right)^2 =  
\sum_{k=1}^{n} \Bigl(4(1- f_k)^2 \big( \wh{\rho}_{k} - \rho^{*} \big)^2 +\zeta_k^2 -4  (1- f_k) \zeta_k \big( \wh{\rho}_{k} - \rho^{*} \big)\Bigr). 
\end{equation}
Since $f_n$ goes to zero almost surely, we have
$$
\sum_{k=1}^{n} \zeta_k^2  =  O(1) +O\left( \sum_{k=1}^{n} f_k^2 \right)=O(1)+ o\left( \sum_{k=1}^{n} f_k \right)
= o (\log n) \hspace{1cm}\text{a.s.} 
$$
Hence, we infer from \eqref{QSLRHO} and \eqref{DECOQSLDW2} that
\begin{equation*}  
\lim_{n\rightarrow \infty} \frac{1}{\log n} \sum_{k=1}^{n} \left( \wh{D}_{k} - D^{*} \right)^2= 4\sigma^2_\rho= \sigma^2_D
\hspace{1cm} \textnormal{a.s.}
\end{equation*}
Furthermore, the law of iterated logarithm \eqref{LILDW} immediately follows from \eqref{LILRHO} and \eqref{CVGFN}, which
completes the proof of Theorem \ref{THM_LILDW}.
\hfill
$\mathbin{\vbox{\hrule\hbox{\vrule height1ex \kern.5em\vrule height1ex}\hrule}}$

\subsection*{}
\begin{center}
{\bf C.4. Proof of Theorem \ref{THM_DWTEST}.}
\end{center}

We shall now establish the asymptotic behavior associated to our Durbin-Watson statistical test. It follows from
the identity $\widetilde{\theta}_n=\wh{\theta}_n + \wh{\rho}_n - \rho_0$ and
\eqref{DEFRHOT} that \\
\begin{eqnarray*}
\wh{\rho}_n - \widetilde{\rho}_n &=& \wh{\rho}_n - \frac{\rho_0 \widetilde{\theta}_n (\widetilde{\theta}_n + \rho_0)}{1 + \rho_0 \widetilde{\theta}_n}, \\
&=& \frac{\wh{\rho}_n + \rho_0 \widetilde{\theta}_n\wh{\rho}_n-\rho_0 \widetilde{\theta}_n (\widetilde{\theta}_n + \rho_0)}{1 + \rho_0 \widetilde{\theta}_n}, \\
&=& \frac{\wh{\rho}_n + \rho_0 \widetilde{\theta}_n(\wh{\rho}_n-\widetilde{\theta}_n - \rho_0)}{1 + \rho_0 \widetilde{\theta}_n}, \\
&=& \frac{\wh{\rho}_n - \rho_0 \widetilde{\theta}_n\wh{\theta}_n}{1 + \rho_0 \widetilde{\theta}_n}.
\end{eqnarray*}
Hence, if $\wh{\gamma}_n = \wh{\rho}_n - \theta \rho_0 \wh{\theta}_n$ and $\wh{\delta}_n = \wh{\theta}_n + \wh{\rho}_n - \theta - \rho_0$, 
we find that
\begin{equation}
\label{DELTARHON}
\wh{\rho}_n - \widetilde{\rho}_n = \frac{\wh{\rho}_n - \rho_0 \wh{\theta}_n(\wh{\theta}_n + \wh{\rho}_n - \rho_0)}{1 + \rho_0 \widetilde{\theta}_n} = \frac{\wh{\gamma}_n - \rho_0 \wh{\theta}_n \wh{\delta}_n }{1 + \rho_0 \widetilde{\theta}_n}.
\end{equation}
Denote
\begin{equation*}
\theta^{*}_0= \frac{\theta + \rho_0}{1 + \theta\rho_0}
\hspace{1cm}\text{and}\hspace{1cm}
\rho^{*}_0= \frac{\theta \rho_0(\theta + \rho_0)}{1 + \theta\rho_0}.
\end{equation*}
Since $\rho^{*}_0=\theta \rho_0 \theta^{*}_0$ and $\theta^{*}_0+\rho^{*}_0=\theta + \rho_0$, we obtain that
$\wh{\gamma}_n = \wh{\rho}_n - \rho^{*}_0- \theta \rho_0 (\wh{\theta}_n- \theta^{*}_0)$ and 
$\wh{\delta}_n = \wh{\theta}_n + \wh{\rho}_n - \theta^{*}_0-\rho^{*}_0$.
Consequently, we deduce from \eqref{DELTARHON} that
\begin{equation}
\label{DELTARHONN}
\wh{\rho}_n - \widetilde{\rho}_n = 
\frac{a(\wh{\theta}_n- \theta^{*}_0)+b(\wh{\rho}_n - \rho^{*}_0)}{1 + \rho_0 \widetilde{\theta}_n} 
-\frac{\rho_0 \wh{\delta}_n (\wh{\theta}_n- \theta^{*}_0) }{1 + \rho_0 \widetilde{\theta}_n}
\end{equation}
where $a=-\rho_0( \theta + \theta^{*}_0)$ and $b= 1 - \rho_0 \theta^{*}_0$. On the other hand, it follows
from \eqref{DECODW} that
\begin{eqnarray}
\wh{D}_n - \widetilde{D}_n &=& 2(1-f_n)\big(1-\wh{\rho}_n\big)+ \xi_n -2\big(1 -\widetilde{\rho}_n\big),  \notag\\
&=& -2\big(\wh{\rho}_n -\widetilde{\rho}_n\big) -2 f_n \big( 1 - \wh{\rho}_n \big) + \xi_n, \notag\\
&=& -2\big(\wh{\rho}_n -\widetilde{\rho}_n\big)  + \Delta_n
\label{DELTADW}
\end{eqnarray}
where $\Delta_n= \xi_n -2 f_n \big( 1 - \wh{\rho}_n \big)$. Therefore, \eqref{DELTARHONN} together with \eqref{DELTADW} imply that
\begin{equation}  
\label{DELTACLTDW}
\sqrt{n} \big( \wh{D}_n - \widetilde{D}_n \big)= - \frac{2w^{\prime} W_n}{1 + \rho_0 \widetilde{\theta}_n}
+ \frac{2 \rho_0 \wh{\delta}_n v^{\prime} W_n }{1 + \rho_0 \widetilde{\theta}_n} 
 + \sqrt{n}\Delta_n
\end{equation}
where $v$ and $w$ are the vectors of $\dR^2$ given by $v^{\,\prime}=(1, 0)$, $w^{\,\prime}=(a, b)$ and
\begin{equation*}
W_n= \sqrt{n}\begin{pmatrix}
\wh{\theta}_{n} - \theta^{*}_0  \\
\wh{\rho}_{n} - \rho^{*}_0
\end{pmatrix}.
\end{equation*}
We already saw by \eqref{CVGFN} that
\begin{equation}
\label{CVGDELTAN}
\lim_{n\rightarrow \infty} \sqrt{n}\Delta_n =  0
\hspace{1cm}\text{a.s.} 
\end{equation}
Moreover, as $|\wh{\delta}_n | \leq |\wh{\theta}_n  - \theta^{*}_0| + | \wh{\rho}_n-\rho^{*}_0|$, the almost sure rates 
of convergence given by \eqref{ASRATETHETA} and \eqref{ASRATERHO} ensure that under the null hypothesis $\cH_0$,
\begin{equation*}
\Bigl| \wh{\delta}_n v^{\prime} W_n \Bigr| = O \left( \frac{\log \log n}{\sqrt{n}}\right) \hspace{1cm}\text{a.s.}
\end{equation*}
leading to
\begin{equation}
\label{CVGDELTANBIS}
\lim_{n\rightarrow \infty} \wh{\delta}_n v^{\prime} W_n =  0
\hspace{1cm}\text{a.s.} 
\end{equation}
Consequently, it follows from the joint asymptotic normality \eqref{CLTTHETARHO} together 
with Slutsky's lemma, \eqref{DELTACLTDW}, \eqref{CVGDELTAN} and \eqref{CVGDELTANBIS} that under the null hypothesis $\cH_0$,
\begin{equation}
\label{DELTACLTDWFIN}
\sqrt{n} \left( \wh{D}_{n} - \widetilde{D}_n \right) \liml \cN ( 0, \tau^2)
\end{equation}
with
$$\tau^2=\frac{4}{(1+ \rho_0 \theta)^2} w^{\prime} \Gamma w$$ 
where the covariance matrix $\Gamma$ is given by \eqref{DEFGAMMA}.
We recall from Remark \ref{REM_GAMMA} that $\Gamma$ is invertible as soon as $\theta \neq - \rho$, which implies that $\tau^2 >0$.
In addition, we obtain from \eqref{DEFGAMMAN} and \eqref{DEFTAUN} that
\begin{equation}
\label{ASCVGTAUN}
\lim_{n\rightarrow \infty} \wh{\tau}_n^{\, 2}= \tau^2
\hspace{1cm}\text{a.s.} 
\end{equation}
Finally, we deduce from \eqref{DELTACLTDWFIN}, \eqref{ASCVGTAUN} and once again from Slutsky's lemma that
under the null hypothesis $\cH_0$,
\begin{equation*}
\frac{\sqrt{n}}{\wh{\tau}_n} \left( \wh{D}_{n} - \widetilde{D}_n \right) \liml \cN ( 0, 1)
\end{equation*}
which obviously implies \eqref{DWTESTH0}. It remains to show that under the alternative hypothesis $\cH_1$,
our test statistic goes almost surely to infinity. We already saw from \eqref{ASCVGRHO} that
\begin{equation}  
\label{ASCVGRHOFIN}
\lim_{n\rightarrow \infty} \wh{\rho}_{n} = \frac{\theta \rho( \theta + \rho)}{1+ \theta \rho} \hspace{1cm} \text{a.s.}
\end{equation}
Moreover, as $\widetilde{\theta}_n$ converges almost surely to $\theta + \rho - \rho_0$, we obtain that
\begin{equation}  
\label{ASCVGRHOTILDEFIN}
\lim_{n\rightarrow \infty} \widetilde{\rho}_{n} = \frac{\rho_0(\theta + \rho)( \theta + \rho - \rho_0)}{1+ \rho_0(\theta +\rho - \rho_0)} \hspace{1cm} \text{a.s.}
\end{equation}
Hence, it follows from \eqref{ASCVGRHOFIN} and \eqref{ASCVGRHOTILDEFIN} that
\begin{equation}  
\label{ASCVGDWH1FIN}
\lim_{n\rightarrow \infty} \big( \wh{\rho}_{n} -\widetilde{\rho}_{n} \big)=
 \frac{(\theta + \rho)( \theta - \rho_0)( \rho - \rho_0)}{(1+ \theta \rho)(1+\rho_0(\theta +\rho - \rho_0))} \hspace{1cm} \text{a.s.}
\end{equation}
Under the alternative hypothesis, this limit is equal to zero if and only if  $\theta = \rho_0$ or $\theta= - \rho$. 
However, these particular cases are already excluded from the study of $\cH_1$.
Consequently, under $\cH_1$, we deduce from \eqref{ASCVGDWH1FIN} that
\begin{equation*}  
\lim_{n\rightarrow \infty} n \big( \wh{\rho}_{n} -\widetilde{\rho}_{n} \big)^2= + \infty \hspace{1cm} \text{a.s.}
\end{equation*}
which, via \eqref{DELTADW}, clearly leads to \eqref{DWTESTH1}, completing the proof of Theorem \ref{THM_DWTEST}.
\hfill
$\mathbin{\vbox{\hrule\hbox{\vrule height1ex \kern.5em\vrule height1ex}\hrule}}$


\nocite{*}

\bibliographystyle{acm}
\bibliography{BPDW2011}

\end{document}